\documentclass[a4paper]{amsart}

\pdfoutput=1

\usepackage[utf8]{inputenc}
\usepackage{amsthm, amssymb, amsmath, amsfonts, mathrsfs}
\usepackage[colorlinks=true, pdfstartview=FitV, linkcolor=blue, citecolor=blue, urlcolor=blue,pagebackref=false]{hyperref}

\usepackage{microtype}

\newtheorem{thm}{Theorem}[section]
\newtheorem{prop}[thm]{Proposition}
\newtheorem{lem}[thm]{Lemma}

\theoremstyle{remark}
\newtheorem{rem}[thm]{Remark}

\renewcommand{\le}{\leqslant}
\renewcommand{\ge}{\geqslant}
\renewcommand{\subset}{\subseteq}
\newcommand{\mcl}{\mathcal}
\newcommand{\msf}{\mathsf}
\newcommand{\B}{\mathbb{B}}
\newcommand{\e}{\mathbf{e}}
\newcommand{\E}{\mathbb{E}}

\renewcommand{\L}{\mathcal{L}}
\newcommand{\N}{\mathbb{N}}
\newcommand{\na}{\nabla}
\newcommand{\Ll}{\left}
\newcommand{\Rr}{\right}
\newcommand{\1}{\mathbf{1}}
\newcommand{\R}{\mathbb{R}}

\newcommand{\Z}{\mathbb{Z}}
\newcommand{\z}{^{\circ}}
\renewcommand{\P}{\mathbb{P}}
\newcommand{\p}{^{(p)}}
\newcommand{\pp}{^{(\ov{p})}}
\newcommand{\ppp}{^{(\ov{p}+p)}}
\newcommand{\ppvp}{^{(\ov{p},p)}}
\newcommand{\ppep}{^{(\ov{p}, e^+)}}
\newcommand{\ppem}{^{(\ov{p}, e^-)}}
\newcommand{\ov}{\overline}
\newcommand{\un}{\underline}
\newcommand{\td}{\tilde}
\newcommand{\eps}{\varepsilon}
\renewcommand{\d}{{\mathrm{d}}}

\newcommand{\Ah}{A_\mathrm{hom}}
\newcommand{\dr}{\partial}
\renewcommand{\o}{\omega}
\newcommand{\on}{[\un{\o}]_n}
\renewcommand{\O}{\Omega}
\newcommand{\tp}{^{(\td{p})}}

\newcommand{\C}{\mathsf{C}}

\DeclareMathOperator*{\osc}{Osc}
\DeclareMathOperator{\arsinh}{arsinh}

\numberwithin{equation}{section}

\title[Expansion of homogenized coefficients]{First-order expansion of homogenized coefficients under Bernoulli perturbations}
\author{Jean-Christophe Mourrat}

\address{EPFL, institut de mathématiques, station 8, 1015 Lausanne, Switzerland}

\begin{document}

\begin{abstract}

Divergence-form operators with stationary random coefficients homogenize over large scales. We investigate the effect of certain perturbations of the medium on the homogenized coefficients. The perturbations considered are rare at the local level, but when occurring, have an effect of the same order of magnitude as the initial medium itself. The main result of the paper is a first-order expansion of the homogenized coefficients, as a function of the perturbation parameter. 

\bigskip


\medskip


\end{abstract}
\maketitle
%
%
%
%
%
%
%
%
\section{Introduction}
\label{s:intro}

Consider a divergence-form operator whose coefficients are random. The randomness of the coefficients model the small-scale irregularity of a medium. If the distribution of the coefficients is stationary and ergodic (and if some ellipticity condition holds), then this random operator can, over large scales, be replaced by an effective operator with constant \emph{homogenized} coefficients. 


The aim of this paper is to study the effect of certain perturbations of the medium on the homogenized coefficients. The perturbations considered are small only in the sense that locally, the medium is perturbed with small probability; but where a perturbation occurs, the change is of the same order of magnitude as the medium itself. This type of perturbation may be called a \emph{Bernoulli perturbation}. Our main purpose is to prove a first-order expansion of the homogenized coefficients, as a function of the perturbation parameter, under conditions of short-range correlations and uniform ellipticity of the medium.

\medskip 

There are at least two important motivations behind this problem. The first concerns the numerical approximation of the homogenized coefficients. Although several techniques for doing so have been identified and analysed \cite{cemracs}, it remains a computationally expensive task, even in low dimensions. It has thus been proposed in \cite{analeb1} to study more efficient techniques that would apply to \emph{weakly random} media, that is, random perturbations of a periodic environment. First-order expansions (as a function of the perturbation parameter) have been proved in \cite{analeb3}, but only for specific types of perturbations that do not include Bernoulli perturbations. Yet, Bernoulli perturbations are arguably the most natural modelling assumption for typical disordered media like composite materials (as an example, see \cite{analeb2} for a cross-section of a composite material used in the aeronautics industry). In \cite{analeb2,analeb3}, conjectures are formulated concerning the expansion of the homogenized coefficients for such perturbations, which are backed by a formal derivation and numerical evidence. 

A second motivation is related to percolation. For this model, each edge of $\Z^d$ is independently removed with probability $p$, and kept otherwise. There exists a critical probability $p_c \in (0,1)$ such that the remaining graph has a unique infinite connected component if $p < p_c$, and has only finite connected components if $p > p_c$. Over the last years, the understanding of two-dimensional percolation close to criticality progressed tremendously, in particular through the rigorous derivation of the values of several critical exponents (see for instance \cite{smwe}). Yet, to the best of my knowledge, it is not known (for any $d \ge 2$) how the homogenized conductivity of the percolation cluster behaves as $p$ approaches $p_c$ from below (see \cite{hug} for a nice review of the problem). Understanding the effect of small perturbations of the medium on the homogenized conductivity seems to be a first necessary step towards the resolution of this problem.

\medskip 

Of the two motivations above, the first is formulated for differential operators, while the second is inherently discrete. In what follows, we focus on a discrete model. Apart from this difference, the results presented here give a proof of the first-order expansion of the homogenized coefficients conjectured in \cite{analeb1, analeb2, analeb3}. It extends it in the sense that the non-perturbed environment considered here need not be periodic, and that the expansion is obtained around every value of the perturbation parameter. 
We work under a condition of uniform ellipticity, which is obviously not satisfied in the case of percolation. The present paper will hopefully lay the basis for an extension to the case of percolation, as well as to higher-order expansions.

\medskip

On a heuristic level, the first-order expansion of the homogenized coefficients can be guessed as follows. The homogenized coefficients can be expressed in terms of the \emph{corrector}, whose defining equation is posed on the whole space. It is well-known that finite-volume approximations of the corrector yield consistent approximations as the volume tends to infinity. One can easily derive a first-order expansion of the corrector defined over a given finite volume element (or any reasonable function of it) as the perturbation parameter tends to $0$. Indeed, in this limit, one can assume at first order that there is at most one location that is perturbed. One then gets a formula for the first-order expansion of the homogenized coefficients by formally interchanging the ``infinite-volume'' and the ``small perturbation parameter'' limits.

We will see here that this informal reasoning can be made rigorous, and thus provide us with a first-order expansion of the homogenized coefficients. The main point is to quantify errors when localizing the problem over a finite volume, say of side length~$N$, and then choose $N$ as a function of the perturbation parameter $p$. It is clear that, for this strategy to make sense at all, we need the box to contain some perturbed locations, so we should have $N^d \gg p^{-1}$. A more refined heuristic consists in observing that a random walk evolving in a box of size $N$ sees only of order $N^2$ sites, so that we should in fact need $N^2 \gg p^{-1}$. Using parabolic/elliptic regularity theory, the averaged estimates on the gradients of the Green function of \cite{deldeu}, and the localization error estimates on the corrector due to \cite{glotto, glotto2}, we will see that the argument can be made rigorous if we choose $N^2 = p^{-(1+\eps)}$ for some small $\eps > 0$. (One should think of $N$ as being $\mu^{-1/2}$ in the notation introduced below.) 

Of these three main ingredients, only the estimates of \cite{glotto,glotto2} require the assumption on the short range of correlations. At bottom, what is required for these estimates to hold is a spectral gap inequality for the Glauber dynamics, which makes so-called ``vertical derivatives'' come in (see \cite{gno} and Remark~\ref{r:corr} below for more on this). 

In order to extend the present results to the case of continuous operators, the only missing main ingredient is an analogue of the results of \cite{glotto, glotto2}. Results in this direction have been announced by the authors. Concerning the extension to percolation, it is noteworthy that a probabilistic version of the Harnack inequality was proved in \cite{barlow}, while part of the estimates of \cite{deldeu} have been extended to percolation clusters in \cite{benjdum}. 

In our discrete, uniformly elliptic setting, stronger results on the Green function and its gradients have been obtained very recently in \cite{marott}. We will see in the last section how to use them to get a sharp control of the error term in the expansion.

\medskip

Let us now give a brief survey of related works, outside of the previously mentioned \cite{analeb1,analeb2,analeb3}.

In some cases, asymptotic formulas for the homogenized coefficients in the regime of high dilution of the perturbations have been known and used by physicists since the 19th century, and are variously known as the Clausius-Mossotti, Lorentz-Lorenz, Maxwell, or Rayleigh formulas. The setting is that of a continuous and homogeneous medium with highly diluted spherical inclusions. In this context, the general expression for the first-order correction that will be obtained here in terms of correctors becomes an explicit formula \cite{maxwell,pap}. Vanishingly small spherical inclusions arranged periodically along a fixed square or cubic lattice within a homogeneous medium have been investigated with precision in \cite{ray}, where a high-order expansion is obtained. Surprisingly, in dimension 3, the expansion involves fractional powers of the volume fraction $p$ of the small spherical inclusions, the smallest fractional power with a non-zero coefficient being $p^{4+1/3}$ (see \cite[(60) and (64)]{ray}). In particular, the homogenized coefficients can be differentiated four times with respect to $p$, but no more. Note that in \cite{ray}, every periodic cell is perturbed by the insertion of a small inclusion, while our present approach would cover the different situation where each periodic cell is perturbed by some fixed inclusion, with small probability. This difference should be irrelevant as far as the first order of the expansion is concerned, but not so for higher orders. In our context of random perturbations, it is an open question whether the homogenized coefficients are infinitely differentiable functions of the perturbation parameter. 

Clausius-Mossotti formulas have been the subject of an impressive amount of work in the physics literature, of which we simply quote \cite{zuzbre}, where results of \cite{ray} are extended to higher-order corrections. In the mathematics literature, the most important result to date on this problem (that I know of) is certainly \cite{alm}. Under some conditions, the Clausius-Mossotti formula is proved (although the results obtained there do not directly relate to the homogenized coefficients), with an error bound of order $p^{3/2}$, where $p$ is the volume fraction of the perturbation. In our context, we will see that the first-order expansion of the homogenized coefficients holds with an error term of order $o(p^{2-\eta})$, for every $\eta > 0$. Earlier investigations include \cite{koz,bermit}.

The problematic considered in \cite{bll} is similar in spirit to the present one. In this work, a different type of perturbation, based on random deformations of the geometry of the medium, has been investigated, and a first-order expansion has been obtained, see \cite[Theorem~3.2]{bll}.

The problem of showing the regularity of diffusion coefficients has also been explored for others models, and in particular for that of a tagged particle in the simple symmetric exclusion process. It is proved in \cite{lov} that the effective diffusivity for this model is an infinitely differentiable function of the density of particles (see also \cite{ber,lov2,bel,sue,nag1,nag2,nag3} for generalizations). The approach followed there relies on a particular duality structure of the process. In our present context, this duality has been investigated in \cite{komoll,cudkom} under the additional assumption that the random coefficients take only two possible values (although no regularity result on the diffusion coefficients was given there). In contrast, the approach of the present paper does not involve any duality structure, and (as a consequence) is not restricted to random coefficients taking only a finite number of possible values.


\medskip

In the next section, we introduce some definitions and notations, and then explain in more detail the approach taken up and the organization of the rest of the paper.

%
%
%
%
%
%
%
%
\section{Definitions and notations}
\label{s:defs}

\subsection{A reminder on homogenization}
We view $\Z^d$ ($d \ge 2$) as a graph, with its usual nearest-neighbour structure, and write $\B$ for the set of (non-oriented) edges. We write $x\sim y$ if $x$ and $y \in \Z^d$ are neighbours.
Let $(\omega_e)_{e \in \B}$ be i.i.d.\ random variables such that almost surely, 
$$
c_- \le \omega_e \le c_+ ,
$$ 
where $0 < c_- < c_+ < +\infty$ are constants, henceforth called the \emph{ellipticity constants}. We call $\omega = (\omega_e)_{e \in \B}$ an \emph{environment}, and write $\Omega = [c_-,c_+]^{\B}$ for the set of environments. We call $\omega_e$ the \emph{conductance} of the edge $e$. We write $\P$ for the law of~$\omega$, and $\E$ for the associated expectation. Note that $\Z^d$ acts on $\Omega$ by the translations $(\theta_z)_{z \in \Z^d}$ defined by $(\theta_z \ \omega)_{x,y} = \omega_{x+z,y+z}$.

Let $(\e_1, \ldots, \e_d)$ be the canonical basis of $\R^d$. For every $x \in \Z^d$, let $A(x,\omega) = A^\o(x)$ be the $d$-dimensional diagonal matrix
$\textrm{diag}(\omega_{x,x+\e_1}, \ldots,\omega_{x,x+\e_d})$. We may keep the variable $\omega$ implicit in the notation. The operator whose homogenization properties are investigated here is $-\na^* \cdot A \na$, which acts on functions $f : \Z^d \to \R$ by
$$
-\na^* \cdot A \na f (x) = \sum_{y \sim x} \omega_{x,y} (f(y)-f(x)) \qquad (x \in \Z^d).
$$
Here and below, we write $\na f$ to denote the forward gradient,
$$
\nabla f(x)=\left[  
\begin{array}{c}
f(x+\e_1)-f(x) \\
\vdots\\
f(x+\e_d)-f(x)
\end{array}
\right],
$$
and for $F = (F_1,\ldots,F_d): \Z^d \to \R^d$, we write $\na^* \cdot F$ for the backward divergence, 
$$
\na^* \cdot F(x) = \sum_{i = 1}^d \Ll(F_i(x) - F_i(x-\e_i)\Rr).
$$

Homogenization refers to the fact that there exists a \emph{constant} symmetric matrix $\Ah$ such that the solution operator of $-\na^* \cdot A(\cdot/\eps) \na$ converges, as $\eps \to 0$, to the solution operator of the differential operator $-\na \cdot \Ah \na$. Moreover, the homogenized matrix $\Ah$ can be characterized in terms of a function called the corrector, which we now proceed to define. 

We say that a function $\psi : \Z^d \times \Omega \to \R$ is \emph{stationary} if $\psi(x,\omega) = \psi(0,\theta_x \ \omega)$. Let $\xi \in \R^d$ be a vector which will be kept fixed throughout this paper.  The \emph{corrector} $\phi$ in the direction $\xi$ is the unique function: $\Z^d \times \Omega \to \R$ such that $\na \phi$ is stationary, $\E[\na \phi] = 0$, $\phi(0) = 0$ and 
\begin{equation}
\label{defphip}
-\na^* \cdot A(\xi + \na \phi)(x,\omega) = 0 \qquad (x \in \Z^d,\  \P\text{-a.e. } \omega).
\end{equation}
There are several ways to define the homogenized matrix in terms of the corrector, but the most useful one for our present purpose is the property that
\begin{equation}
\label{defAh}
\xi \cdot \Ah \xi = \E[\xi \cdot A(\xi + \na \phi)].
\end{equation}
In the right-hand side above, we keep implicit the fact that the quantity under the expectation is evaluated at $(0,\omega)$. 
%
Proofs of existence and uniqueness of $\phi$ as stated above, and of the formula \eqref{defAh}, can be found for instance in \cite[Theorem 3 and (3.17)]{kuen}, 
or in the continuous setting, in \cite[Theorem~2]{papvar}.

\subsection{Bernoulli perturbations}
We now introduce Bernoulli perturbations. We give ourselves a second family $({\o}^{(1)}_e)_{e \in \B}$ of i.i.d.\ random variables such that almost surely,
$$
c_- \le {\o}^{(1)}_e \le c_+.
$$
For every $p \in [0,1]$, we also give ourselves a family $(b\p_e)_{e \in \B}$ of independent Bernoulli random variables of parameter $p$, independent of $(\omega_e,{\o}^{(1)}_e)_{e \in \B}$. We now understand that $\P$ denotes the joint law of $\un{\o} := (\omega_e,{\o}^{(1)}_e,(b\p_e)_{p \in [0,1]})_{e \in \B}$, so that for instance, we have $p = \P[b_e\p = 1]$. We write $\un{\Omega}$ for the set of possible values taken by~$\un{\o}$. The group $\Z^d$ acts on $\un{\Omega}$ by translations, and we keep writing $(\theta_x)_{x \in \Z^d}$ to denote this action. 

We let $\omega\p_e = (1-b\p_e)\ \omega_e + b\p_e \ {\o}^{(1)}_e$. Note that the notation is consistent when $p = 1$, and that if we denote the law of $\o\p_e$ by $\nu\p$, then $\nu\p = (1-p) \nu^{(0)} + p \nu^{(1)}$. With each $p \in [0,1]$ is thus associated the perturbed environment $\omega\p = (\omega\p_e)_{e \in \B}$. The environment $\omega\p$ shares the same properties as the environment $\omega$, and thus we can define $A\p, \Ah\p$ and $\phi\p$ in the same way as $A, \Ah$, and $\phi$ respectively, but with $\omega$ replaced by $\omega\p$. Throughout this article, for notational convenience, we will replace the exponent $^{(0)}$ by simply $^\circ$, so that for instance, $A\z = A^{(0)}$, $\Ah\z = \Ah^{(0)}$, $\phi\z = \phi^{(0)}$, and so on (we think of $A\z$, $\phi\z$, etc.\ as functions of $\un{\o}$, which makes them formally different from $A$, $\phi$, etc.\ introduced before). 

\medskip

Our main goal is to show that the homogenized matrix $\Ah\p$, as a function of~$p$, is differentiable, and to find an explicit formula for the derivative.
Heuristically, one may expect that a linear approximation in \eqref{defAh} gives the correct first-order approximation, that is, as $p$ tends to $0$,
\begin{equation}
\label{e:linearapprox}
\xi \cdot \Ah\p \xi = \xi \cdot \Ah\z \xi + p \sum_{e \in \B} \E[\xi \cdot A^e(\xi + \na \phi^e) - \xi \cdot A\z(\xi + \na \phi\z)] + o(p),
\end{equation}
where $A^e$ and $\phi^e$ are for the environment perturbed at the edge $e$ what $A$ and $\phi$ are for the unperturbed environment. In the expectation above and throughout this paper, it is kept implicit that the functions are evaluated at $(0,\un{\o})$. Our aim is to justify this heuristic. Note that the meaning of the sum in \eqref{e:linearapprox} is not clear, since it is not absolutely summable. 

\medskip

\noindent \textbf{Organization of the paper.} In section~\ref{s:green}, we recall classical results on the decay at infinity of the Green function and its gradients. A difficulty with the formula \eqref{defAh} is that the corrector is a very non-local function. It is convenient to introduce a localized version of the corrector, obtained by adding a zero-order term in equation~\eqref{defphip}. In section~\ref{s:solve}, we give a simple criterion for the existence and uniqueness of solutions of elliptic equations with a zero-order term. In section~\ref{s:loc}, we recall results quantifying the accuracy of the approximation by this localized corrector. Although the (localized) corrector depends non-linearly on the values taken by the environment~$\o\p$, we will show in section~\ref{s:linear} that it is, at first order, close to its linear approximation in the limit of small~$p$. After having derived several useful convergence results related to stationary elliptic equations in section~\ref{s:station}, we state and prove Theorem~\ref{t:main0} in section~\ref{s:expansion0}, which is a rigorous version of \eqref{e:linearapprox}. We show in section~\ref{s:periodiz} that localization by a zero-order term in \eqref{defphip} can be replaced by periodization of the medium, thus providing us with alternative descriptions of the derivative of $\Ah\p$ at $p = 0$. The goal of the last two sections is to generalize Theorem~\ref{t:main0} in two directions. In section~\ref{s:expansion}, we give the asymptotic expansion of $\Ah\ppp$ as $p$ tends to $0$, for every $\ov{p} \in [0,1]$. Finally, using the recent results of \cite{marott}, we prove in section~\ref{s:error} that the error term in the expansion is $o(p^{2-\eta})$ for every $\eta > 0$.


\begin{rem}
\label{r:corr}
Although we work throughout under the assumption that the medium is made of i.i.d.\ random variables, the method can be generalized to some (stationary) weakly correlated environments. For this reason, we will not use some properties specific to the i.i.d.\ case (as e.g.\ the fact that the homogenized matrices $\Ah\p$ are multiples of the identity). The bottleneck in generalizing the present approach to more general distributions lies with the results of \cite{glotto,glotto2} recalled in section~\ref{s:loc}. In \cite{gno}, these results were shown to hold as soon as the law of the medium satisfies a spectral gap inequality with respect to the Glauber dynamics. This last condition is more general than the i.i.d.\ assumption, although very far from covering arbitrary ergodic media. The estimates of \cite{glotto,glotto2} are not expected to hold for arbitrary ergodic media however. This can be checked directly in dimension~$1$, or can be derived from the precise quantitative results of \cite{balgar} on homogenization in very correlated environments. (Outside of this remark, recall that we always assume $d \ge 2$ here.)
\end{rem}


\noindent \textbf{Notation.} We define $a \vee b=\max(a,b)$. We write $|\cdot|$ for the $L^2$ norm on $\R^d$. This norm induces an operator norm on $d \times d$ matrices, which we also denote by $|\cdot|$. For a function of several variables, e.g.\ $f : \R \times \Z^d \times \Z^d \to \R$, we write $\na_2 f$ and $\na_3 f$ to denote the (forward) gradient with respect to the second and to the third variable respectively. For $i,j \in \{2,3\}$, we write $\na_i \na_j f$ to denote the $d \times d$ matrix whose columns are 
$
\na_i F_1, \ldots, \na_i F_d,
$
where $[F_1, \ldots, F_d]^{\msf{T}} = \na_j f$. Note that if $v \in \R^d$ is a fixed vector, and $h = \na_j f \cdot v$, then $\na_i h = (\na_i \na_j f) v$. 

The value of a constant denoted by $c$ or $\td{c}$ may change from one occurrence to another, but depends only on the dimension and the ellipticity constants.
%
%
%
%
%
%
%
%
\section{Parabolic and elliptic regularity}
\label{s:green}


The aim of this section is to gather known results on the regularity of the heat kernel and Green function associated with divergence-form operators. We begin by defining the heat kernel. For every fixed $\omega \in \Omega$ and $x \in \Z^d$, let $(q^\o(t,x,y))_{(t,y) \in \R_+ \times \Z^d}$ be the unique bounded function satisfying
$$
\left\{
\begin{array}{l}
\dr_t q^\omega(t,x,y) = -\na^* \cdot A^\o(y) \na q^\omega(t,x,y) \qquad ((t,y) \in \R_+ \times \Z^d), \\
q^\omega(0,x,y) = \1_{y = x} \qquad (y \in \Z^d),
\end{array}
\right.
$$
where we understand $\na^*$ and $\na$ as acting on the $y$ variable. Note that $q^\o(t,x,y) = q^\o(t,y,x)$. Let $q^*(t,x) = q^{\omega^*}(t,0,x)$, where $\omega^*$ is the constant environment such that $\omega^*_e = 1$ for every $e \in \B$. 

\begin{prop}[pointwise control of the heat kernel \cite{stzh,del}]
\label{parabolic}
There exist constants $c, k > 0, \alpha > 0$ such that for every $\omega \in \Omega$, $t \ge 0$, $x \in \Z^d$ and $i \in \{2,3\}$,
\begin{equation}
\label{parabpoint}
q^\omega(t,0,x) \le c \ q^*(kt,x),
\end{equation}
\begin{equation}
\label{parabpointgrad}
\Ll|\na_i q^\omega(t,0,x)\Rr| \le c \  \frac{q^*(k t,x)}{(1 \vee t)^\alpha}.
\end{equation}
\end{prop}
\begin{proof}
The first part appears in this form in \cite[Section~4]{deldeu}, and the proof given there relies on the closely related \cite[Proposition~3.4]{del} (a similar statement was also obtained in \cite[Lemma~1.9]{stzh}). It also appears that as soon as \eqref{parabpoint} holds for some value of $k$, it holds for every larger value as well (with a different $c$). This will also be clear from the proof of \eqref{parabpointgrad}, to which we now turn.

Inequality \eqref{parabpointgrad} roughly corresponds to \cite[Theorem~1.31]{stzh}. In order to provide the reader with a precise proof, we begin by recalling the following consequence of the parabolic Harnack inequality. For $r > 0$ and $x \in \Z^d$, let
$$
Q(x,r) = [0,4r^2] \times \{ z \in \Z^d : |z-x| \le 2r \},
$$
$$
Q'(x,r) = [3r^2,4r^2] \times \{z \in \Z^d : |z-x| \le r \}.
$$
For a real function $u$ defined on a set $A$, we write
$$
\osc_A u = \sup_A u - \inf_A u.
$$
\begin{prop}[increase of oscillations \cite{stzh}]
\label{p:oscill}
There exists $\lambda > 1$ such that for every $\omega \in \Omega$, $r > 0$ and $x \in \Z^d$, if $u : \R_+ \times \Z^d \to \R$ satisfies
$$
\dr_t u = -\na^* \cdot A(\cdot, \omega) \na u \qquad \text{in } Q(x,r),
$$
then
$$
\osc_{Q(x,r)} u \ge \lambda \osc_{Q'(x,r)} u.
$$
\end{prop}
The proof of Proposition~\ref{p:oscill} can be found in \cite[Lemma~1.30]{stzh}, or can be derived from the parabolic Harnack inequality given in \cite[Theorem~2.1]{del} as in the proof of \cite[Proposition~3.4]{deldeu} (see also \cite[p.\ 188]{del} on the fact that numerical constants appearing in the definitions of $Q$ and $Q'$ are irrelevant, and can thus be chosen as convenient).

We will also need the following result, which we borrow from \cite[Section~4]{deldeu}. The expressions involved are less simple than what one might imagine a priori, due to lattice effects.
\begin{prop}[explicit estimates on $q^*$]
\label{p:estimqetoile}
Let 
$$
\mcl{D}_t(x) = |x| \arsinh\Ll( \frac{|x|}{t} \Rr) + t \Ll( \sqrt{1+\frac{|x|^2}{t^2}} - 1 \Rr).
$$
There exist constants $c_1,c_2,k_1,k_2$ such that for every $t \ge 0$ and $x \in \Z^d$,
$$
\frac{c_1}{(1\vee t)^{d/2}}\exp\Ll(-\mcl{D}_{k_1t}(x)\Rr) \le q^*(t,x) \le \frac{c_2}{(1\vee t)^{d/2}}\exp\Ll(-\mcl{D}_{k_2t}(x)\Rr).
$$
\end{prop}
\begin{rem}
The following alternative expression for $\mcl{D}_t(x)$ is also given in \cite{deldeu}:
$$
\mcl{D}_t(x) = \max_{s} \Ll(s |x| - t (\cosh(s) - 1)\Rr),
$$
which makes it transparent that $t \mapsto \mcl{D}_t(x)$ is decreasing (in the wide sense).
\end{rem}
Let us see how to use these results to prove that \eqref{parabpointgrad} holds. We fix $k$ such that \eqref{parabpoint} holds, and $k_1, k_2$ as given by Proposition \ref{p:estimqetoile}. We let $k_3 > 0$ be some parameter yet to be fixed, and begin by proving that \eqref{parabpointgrad} holds provided $|x| \le k_3 t$. (This is the most relevant case, but lattice effects prevent us from giving a unified proof.)

It follows from Proposition~\ref{p:estimqetoile} that there exists constants $c_1,c_2$ such that if $|x|\le 2 k_3 k t$, then 
\begin{equation}
\label{gaussbounds}
\frac{c_1}{(1 \vee t)^{d/2}} e^{-|x|^2/c_1 t} \le q^*(t,x) \le \frac{c_2}{(1 \vee t)^{d/2}} e^{- |x|^2/c_2 t}.
\end{equation}
Let $x,t$ satisfy $|x| \le k_3 t$, and let $L = \lfloor \log_2(t/2)/2 \rfloor$. For $L < 0$ (i.e.\ $t < 2$), there is nothing to prove since \eqref{parabpoint} holds. Otherwise, for every $l \in \{0, \ldots, L\}$, we consider the cylinders
$$
Q_l =[t-2^{2l},t] \times \{z \in \Z^d : |z-x| \le 2^l\}.
$$
Let $y$ be a neighbour of $x$. Noting that $(t,x)$ and $(t,y)$ both belong to $Q_0$ and applying Proposition~\ref{p:oscill} iteratively, we obtain that
$$
\Ll| q^\omega(t,0,x) - q^\o(t,0,y) \Rr| \le \lambda^{-L} \osc_{Q_L} q^\o(\cdot,0,\cdot) \le \lambda^{-L} \sup_{Q_L} q^\o(\cdot,0,\cdot).
$$
By the definition of $L$, the cylinder $Q_L$ is included in the cylinder
$$
Q = [t/2,t] \times \{z \in \Z^d : |z-x| \le \sqrt{t}\}.
$$
In view of \eqref{parabpoint} and \eqref{gaussbounds}, we thus get that
$$
\Ll| q^\omega(t,0,x) - q^\o(t,0,y) \Rr| \le c \lambda^{-L} \ q^*(k_0 t,x),
$$
for every $k_0 > 0$ large enough, uniformly over $x$ and $t$ satisfying $|x| \le k_3 t$. This is indeed the inequality \eqref{parabpointgrad} by the definition of $L$ and the fact that $\lambda > 1$. 

There remains to justify \eqref{parabpointgrad} for $|x| > k_3 t$. Recall that we still have the freedom to choose $k_3$ as convenient. By \eqref{parabpoint} and Proposition~\ref{p:estimqetoile}, we have
\begin{equation}
\label{eqeq}
q^\o(t,0,x) \le q^*(kt,x) \le \frac{c}{(1\vee t)^{d/2}} \exp\Ll(- \mcl{D}_{k_2kt}(x)\Rr).
\end{equation}
We now claim that choosing $k_3$ sufficiently large, we can ensure that for every $\td{k} \ge 2 k k_2$,
\begin{equation}
\label{ineqmclD}
|x| \ge k_3 t \ \Rightarrow \ \mcl{D}_{k_2kt}(x) \ge t + \mcl{D}_{\td{k} t}(x).
\end{equation}

Let us first see why \eqref{ineqmclD} enables to conclude. It follows from \eqref{eqeq} and \eqref{ineqmclD} that
$$
q^\o(t,0,x) \le e^{-t} \frac{c}{(1\vee t)^{d/2}} \exp\Ll(- \mcl{D}_{\td{k}t}(x)\Rr).
$$
By Proposition~\ref{p:estimqetoile}, this is smaller than
$$
c e^{-t} q^*(\td{k} t / k_1, x),
$$
for every $\td{k} \ge 2 k k_2$ and every $x, t$ satisfying $|x| \ge k_3 t$. This clearly implies \eqref{parabpointgrad} for $|x| \ge k_3 t$. The proof is thus complete, provided we show \eqref{ineqmclD}.

In order to do so, we first note that
$$
t \Ll( \sqrt{1+\frac{|x|^2}{t^2}} - 1 \Rr) = \sqrt{t^2  + |x|^2} - t = \int_0^{|x|} \frac{u}{\sqrt{t^2 + u^2}} \ \d u \le |x|.
$$
The right-hand side of \eqref{ineqmclD} thus holds as soon as 
\begin{equation}
\label{ineq22}
|x| \arsinh\Ll( \frac{|x|}{k_2 k t} \Rr) \ge t + |x| + |x|\arsinh\Ll( \frac{|x|}{2 k_2 k t} \Rr).
\end{equation}
If $k_3$ is sufficiently large (and in particular $\ge 1$), then it is clear that $|x| \ge k_3 t$ implies
$$
\arsinh\Ll( \frac{|x|}{k_2 k t} \Rr) \ge 2 + \arsinh\Ll( \frac{|x|}{2 k_2 k t} \Rr),
$$
which in turn implies \eqref{ineq22}.
\end{proof}
The bound \eqref{parabpointgrad} on the gradient of the heat kernel cannot be improved in general (see for instance \cite[p.\ 364]{deldeu}). However, upper bounds matching the homogeneous case can be recovered if one takes averages over the randomness of the stationary field of conductances. For simplicity, we write $q\p(t,x,y)$ instead of $q^{\o\p}(t,x,y)$. The following result is due to \cite[Theorem~1.4]{connadny} and \cite[Theorem~1.1]{deldeu}.
\begin{prop}[averaged control of the heat kernel \cite{connadny,deldeu}]
There exist constants $c, k > 0$ such that for every $p \in [0,1]$, $t \ge 0$, $x \in \Z^d$ and $i \neq j \in \{2,3\}$, 
\begin{equation}
\label{parabaverage1}
\Ll( \E[|\na_i q\p(t,0,x)|^2]  \Rr)^{1/2} \le c \ \frac{q^*(k t, x)}{\sqrt{1 \vee t}},
\end{equation}
\begin{equation}
\label{parabaverage2}
 \E[|\na_i \na_j q\p(t,0,x)|]   \le c \ \frac{q^*(k t, x)}{{1 \vee t}}.
\end{equation}
\end{prop}
For every $\omega \in \Omega$, $\mu > 0$ and $x,y \in \Z^d$, we define the Green function 
$$
G_\mu^\o(x,y) = \int_0^{+\infty} e^{-\mu t} q^\o(t,x,y) \ \d t.
$$
Note that the function $(G_\mu^\omega(x,y))_{y \in \Z^d}$ satisfies
$$
\mu G^\o_\mu(x,y) - \na^* \cdot A(y,\omega) \na G^\o_\mu(x,y) = \1_{x = y} \qquad (y \in \Z^d),
$$
where $\na^*$ and $\na$ are understood as acting on the $y$ variable. Note also that the Green function is symmetric, that is, $G_\mu^\o(x,y) = G_\mu^\o(y,x)$. The estimates obtained on the heat kernel transfer into estimates on the Green function by integration.


\begin{prop}[pointwise control of the Green function]
There exist $c, \td{c} > 0, \alpha > 0$ such that for every $\mu \in (0,1/2]$, $\omega \in \Omega$, $x \in \Z^d$ and $i \in \{1,2\}$,
\begin{equation}
\label{greenpoint1}
|G^\o_\mu(0,x)| \le 
c \left|
\begin{array}{ll}
\displaystyle{
\log(\mu^{-1}) \ e^{-\td{c} \sqrt{\mu} |x|} 
}	& \text{if } d = 2, \\
\displaystyle{
\frac{1}{(1\vee|x|)^{d-2}} \ e^{-\td{c} \sqrt{\mu} |x|}
}	& \text{if } d \ge 3,
\end{array}
\right.
\end{equation}
\begin{equation}
\label{greenpoint2}
|\nabla_i G^\o_\mu(0,x)| \le \frac{c}{(1\vee|x|)^{d-2+\alpha}} \ e^{-\td{c} \sqrt{\mu} |x|}.
\end{equation}
\end{prop}
\begin{proof}
We begin by justifying \eqref{greenpoint1} for $d \ge 3$. In view of \eqref{parabpoint}, it suffices to bound
\begin{equation}
\label{ineq1}
\int_0^{+\infty} e^{-\mu t} p^*(t,x) \ \d t = \int_0^{|x|} e^{-\mu t} p^*(t,x) \ \d t + \int_{|x|}^{+\infty} e^{-\mu t} p^*(t,x) \ \d t.
\end{equation}
Using \eqref{gaussbounds}, we can bound the second integral in the right-hand side above, up to a constant, by
\begin{equation}
\label{integr1}
\int_0^{+\infty} \frac{e^{-\mu t}}{(1 \vee t)^{d/2}} e^{-|x|^2/c_2 t} \ \d t,
\end{equation}
which, by a change of variables, can be bounded by
$$
|x|^{-(d-2)} \int_0^{+\infty} \frac{e^{-\mu|x|^2 u}}{u^{d/2}} e^{-1/c_2  u} \ \d u.
$$
Moreover,
\begin{eqnarray*}
&& \int_0^{+\infty} \frac{e^{-\mu|x|^2 u}}{u^{d/2}} e^{-1/c_2 u} \ \d u \\
& \le & e^{-\sqrt{\mu} |x| /2c_2}  \int_0^{(\sqrt{\mu} |x|)^{-1}} \frac{e^{-1/2 c_2 u}}{u^{d/2}} \ \d u + e^{-\sqrt{\mu} |x|} \int_{(\sqrt{\mu} |x|)^{-1}}^{+\infty} \frac{e^{-1/c_2 u}}{u^{d/2}} \ \d u \\
& \le & C (e^{-\sqrt{\mu} |x| /2c_2} + e^{-\sqrt{\mu}|x|}).
\end{eqnarray*}
We have thus obtained that for $d \ge 3$, there exists $c,\td{c} > 0$ such that
$$
\int_{|x|}^{+\infty} e^{-\mu t} p^*(t,x) \ \d t \le \frac{c}{|x|^{d-2}} e^{-\td{c} \sqrt{\mu} |x|}.
$$
The first integral in the right-hand side of \eqref{ineq1} can be bounded more easily. Indeed, it follows from Proposition~\ref{p:estimqetoile} that for $t < |x|$, we have $p^*(t,x) \le e^{-\td{c} |x|}$ for some $\td{c} > 0$, and thus
$$
\int_0^{|x|} e^{-\mu t} p^*(t,x) \ \d t \le |x| e^{-\td{c}|x|}.
$$
To summarize, we have thus shown that
$$
\int_0^{+\infty} e^{-\mu t} p^*(t,x) \ \d t \le \frac{c}{|x|^{d-2}} e^{-\td{c} \sqrt{\mu} |x|}.
$$
In order to conclude, it then suffices to see that the Green function is bounded uniformly over $\mu$. But this is clear since $\int p^*(t,0)  \d t$ is finite.

When $d = 2$, one needs to be more careful when considering the integral appearing in \eqref{integr1}. It is straightforward to check that this integral is bounded by a constant times $\log(\mu^{-1})$, uniformly over $x$. If $\sqrt{\mu}|x| \le 1$, then there is nothing more to prove. Else, we split the integral along
$$
\Ll(\int_0^1 + \int_1^{|x|/\sqrt{\mu}} + \int_{|x|/\sqrt{\mu}}^{+\infty} \Rr) \frac{e^{-\mu t}}{1 \vee t} \ e^{-|x|^2/c_2 t} \ \d t.
$$
The first integral is bounded by $e^{-|x|^2/c_2}$, and can thus be discarded. For the second one, we obtain the bound
$$
e^{-\sqrt{\mu}|x|/c_2} \int_1^{+\infty}\frac{e^{-\mu t}}{t} \ \d t \le 2 \log(\mu^{-1}) e^{-\sqrt{\mu}|x|/c_2},
$$
while the third integral can be bounded by
$$
e^{-\sqrt{\mu}|x|/2} \int_{|x|/\sqrt{\mu}}^{+\infty} \frac{e^{-\mu t /2}}{t} \ \d t \le \frac{e^{-\sqrt{\mu}|x|}}{\sqrt{\mu}|x|} ,
$$
and the claim is thus proved. 

Inequality \eqref{greenpoint2} is obtained from \eqref{parabpointgrad}, following the same reasoning as for the case $d \ge 3$ of the proof of \eqref{greenpoint1}.
\end{proof}

We write $G\p$ for $G^{\o\p}$.

\begin{prop}[averaged control of the Green function]
\label{p:greenaverage}
There exists $c, \td{c} > 0$ such that for every $p \in [0,1]$, $\mu > 0$, $x \in \Z^d$ and $i \neq j \in \{1,2\}$,
\begin{equation}
\label{greenaverage1}
\Ll( \E[|\nabla_i G\p_\mu(0,x)|^2] \Rr)^{1/2} \le \frac{c}{(1\vee|x|)^{d-1}} \ e^{-\td{c} \sqrt{\mu} |x|},
\end{equation}
\begin{equation}
\label{greenaverage2}
\E[|\nabla_i\nabla_j G\p_\mu(0,x)|]  \le \frac{c}{(1\vee|x|)^{d}} \ e^{-\td{c} \sqrt{\mu} |x|},
\end{equation}
\end{prop}
\begin{proof}
Recall that Minkowski's integral inequality ensures that for any positive measurable function $f$, 
$$
\E\Ll[ \Ll( \int f(s) \ \d s \Rr)^2\Rr] 
\le \Ll( \int \E[f(s)^2]^{1/2} \ \d s  \Rr)^2.
$$
Applying this observation to 
$$
\E[|\nabla_i G\p_\mu(0,x)|^2] \le \E\Ll[ \Ll(\int e^{-\mu t} |\na_i q\p(t,0,x)| \ \d t\Rr)^2 \Rr],
$$
we get that
$$
\Ll( \E[|\nabla_i G\p_\mu(0,x)|^2] \Rr)^{1/2} \le \int_0^{+\infty} e^{-\mu t} \ \E\Ll[|\na_i q\p(t,0,x)|^2\Rr]^{1/2} \ \d t.
$$
The rest of the proof of \eqref{greenaverage1} is the same as that of \eqref{greenpoint1}, proceeding by integration of the estimate obtained in \eqref{parabaverage1}. The proof of \eqref{greenaverage2} is identical, integrating estimate \eqref{parabaverage2}.
\end{proof}

%
%
%
%
%
%
%
\section{Solving regularized elliptic equations}
\label{s:solve}

We say that a function $f : \Z^d \to \R$ is of \emph{sub-exponential growth} if for every $\delta > 0$, we have $|f(x)| = O(e^{\delta|x|})$ as $|x|$ tends to infinity.

\begin{thm}[Existence and uniqueness of solutions]
\label{t:ex-un}
Let $\mu > 0$, and $f: \Z^d \to \R$ be of sub-exponential growth. For every $\omega \in \Omega$, there exists a unique function $\chi : \Z^d \to \R$ of sub-exponential growth satisfying
\begin{equation}
\label{eqphichi}
\mu \chi - \na^* \cdot A^\o \na \chi = f  \qquad (\text{in } \Z^d).
\end{equation}
Moreover, $\chi$ is given by
\begin{equation}
\label{defphichi}
\chi(x) = \sum_{y \in \Z^d} G^\o_\mu(x,y) \ f(y) \qquad (x \in \Z^d).
\end{equation}
\end{thm}
\begin{proof}
Since \eqref{greenpoint1} ensures exponential decay of the Green function, it is easy to see that the function $\chi$ written in \eqref{defphichi} is well-defined, of sub-exponential growth, and satisfies equation \eqref{eqphichi}. There remains to show uniqueness, and for this, it suffices to consider the case $f = 0$. Let $B_n = \{-n,\ldots,n\}^d$ be the box of size $n$. We have
\begin{equation}
\label{e:unique}
\sum_{x \in B_n} \chi(x)\Ll(\mu \chi(x) - \na^* \cdot A^\o(x) \na \chi(x) \Rr) = 0.
\end{equation}
We want to perform an integration by parts. The difference between 
$$
- \sum_{x \in B_n} \chi(x) \na^* \cdot A^\o(x) \na \chi(x)
$$
and 
\begin{equation}
\label{e:unique2}
\sum_{x \in B_n} \na \chi(x) \cdot A^\o(x) \na \chi(x)
\end{equation}
is bounded by
$$
\sum_{\substack{x \in B_n, y \in B_{n+1} \\ x \sim y}} |\chi(x)| \ |A^\o(y) \na \chi(y)| + |\chi(y)| \ |A^\o(x) \na \chi(x)|,
$$
which, up to a constant, is bounded by
$$
\sum_{x \in B_{n+2} \setminus B_{n-2}} \chi(x)^2.
$$
Hence, combining this estimate with \eqref{e:unique} and the fact that the term in \eqref{e:unique2} is positive, we obtain
$$
\mu \sum_{x \in B_n} \chi(x)^2 \le c \sum_{x \in B_{n+2} \setminus B_{n-2}} \chi(x)^2,
$$
for some constant $c > 0$. In a more condensed form, if we write 
$$
u_n =\sum_{x \in B_n} \chi(x)^2,
$$
we have thus obtained that $\mu \ u_n \le c (u_{n+2} - u_{n-2})$. Since $u_n$ is increasing, we have shown that 
$$
u_{n+2} \ge \Ll(1+\frac{\mu}{c}\Rr)u_{n-2}.
$$
If $u_n$ is not identically zero, then it must grow exponentially fast. But this would contradict the assumption that $\chi$ is of sub-exponential growth. Hence $u_n$ is identically zero, and so is $\chi$.
\end{proof}

%
%
%
%
%
%
%
\section{Approximating the homogenized matrix}
\label{s:loc}
Let $\mu > 0$. For every $p \in [0,1]$ and $\un{\o} \in \un{\O}$, we let $\phi_\mu\p(\cdot,\un{\o}) : \Z^d \to \R$ be the unique function of sub-exponential growth satisfying
\begin{equation}
\label{defphimu}
\mu \phi_\mu\p - \na^* \cdot A\p (\xi + \na \phi_\mu\p) = 0.
\end{equation}
We write $\phi_\mu\z$ for $\phi_\mu^{(0)}$. The introduction of the function $\phi_\mu\p$ is interesting for our purpose as an approximation of $\phi\p$. The former is more localized than $\phi\p$, in the sense that, roughly speaking, the value of $\phi_\mu\p$ at a point depends only on the values of the conductances in a box of size $\mu^{-1/2}$.

A standard energy estimate shows that $\E[|\na \phi_\mu\p|^2]$ is bounded uniformly over $\mu$ and $p$ (we recall that whenever we write an expression like $\E[ |\na \phi_\mu\p|^2 ]$, we understand that the function under the expectation is evaluated at $(0,\un{\o})$). We will need the following much more precise information on the corrector, which is borrowed from \cite[Proposition~2.1]{glotto}.
\begin{thm}[finite moments of the corrector, \cite{glotto}]
\label{t:corr-moments}
For every $k \in \N$, there exist constants $c$ and $r \ge 0$ such that for every $\mu \in (0,1/2]$ and $p \in [0,1]$, 
$$
\E\Ll[ | \phi_\mu\p |^k \Rr] \le 
c \left| 
\begin{array}{ll}
	 \log^r(\mu^{-1}) & \text{if } d = 2, \\
	1  & \text{if } d \ge 3.
\end{array}
\right.
$$
\end{thm}

It is crucial for our subsequent arguments to have a good quantitative estimate on the difference between the homogenized matrix and the approximation obtained through $\phi_\mu\p$. The following result is what we need, and follows from \cite[Theorem~1]{glotto2} (or also from \cite{glotto} and \cite[Proposition~9.1 with $k = 1$]{vardecay}).

\begin{thm}[approximation of $\Ah$ based on $\phi_\mu$, \cite{glotto,vardecay, glotto2}] 
There exist constants $c$ and $r \ge 0$ such that for every $\mu \in (0,1/2]$ and $p \in [0,1]$,
$$
\Ll|\xi \cdot \Ah\p \xi - \E[\xi \cdot A\p(\xi + \na \phi_\mu\p)] \Rr| \le c
\left|
\begin{array}{ll}
	\mu \log^r(\mu^{-1}) & \text{if } d = 2, \\
	\mu  & \text{if } d \ge 3.
\end{array}
\right.
$$
\label{t:approx}
\end{thm}

%
%
%
%
%
%
%
%
\section{Linear approximation of the corrector}
\label{s:linear}

Each edge $e$ can be written uniquely as $(z,z+\e_i)$ for some $z \in \Z^d$ and $1 \le i \le d$. We write $\un{e} = z$. 
We define $C^e : \Z^d \times \un{\Omega} \to \mcl{M}_d$ (the set of $d \times d$ matrices) by $C^e(x,\un{\o}) = 0$ (the matrix whose elements are all $0$) if $x \neq \un{e}$, and $C^e(\un{e},\un{\o})$ to be the diagonal matrix whose diagonal elements are all $0$ except the $i^\text{th}$ one, which is equal to ${\o}^{(1)}_e - \o_e$. As usual, we may keep the variable $\un{\o}$ implicit if clear from the context. 
We let $A^e = A^\circ + C^e$. In words, the environment corresponding to $A^e$ is perturbed only at the edge $e$. More generally, for every $E \subset \B$, we define
\begin{equation}
\label{defAE}
C^E = \sum_{e \in E} C^e \quad \text{ and } \quad A^E = A^\circ + C^E,
\end{equation}
$G_\mu^E$ for the Green function corresponding to the environment given by $A^E$, and $E\p = \{e \in \B : b\p_e = 1\}$. Note that by definition, $A\p = A^{E\p}$ and $G_\mu\p = G_\mu^{E\p}$. For every $E \subset \B$, we let $\phi^E_\mu : \Z^d \times \un{\Omega} \to \R$ be the unique function of sub-exponential growth such that
\begin{equation}
\label{defphiE}
\mu \phi^E_\mu - \na^* \cdot A^E (\xi + \na \phi^E_\mu) = 0,
\end{equation}
as given by Theorem~\ref{t:ex-un}, so that
\begin{equation}
\label{defphiE-green}
\phi^E_\mu(x) = \sum_{y \in \Z^d} G^E_\mu(x,y) \ \na^* \cdot A^E(y) \xi = - \sum_{y \in \Z^d} \na_2 G^E_\mu(x,y)  \cdot A^E(y) \xi,
\end{equation}
where the integration by part in the last equality is justified since $G$ decays exponentially fast, see \eqref{greenpoint1}. We have $\phi\p = \phi^{E\p}$.
We write $\phi^e_\mu$ as a shorthand for $\phi^{\{e\}}_\mu$. 
We also define
\begin{equation}
\label{deftdphi}
\td{\phi}^E_\mu = \phi\z_\mu + \sum_{e \in E} (\phi^e_\mu - \phi\z_\mu)
,
\end{equation}
(we will see shortly that this definition makes sense), and write $\td{\phi}\p_\mu$ for $\td{\phi}^{E\p}_\mu$. The purpose of this section is to prove the following result, which roughly speaking, states that in the limit of diluted perturbations, the perturbed corrector is close to its linear approximation.

\begin{thm}[linear approximation of the corrector]
\label{t:approxcorr}
There exist constants $c$, $\beta > 0$, $\gamma > 0$ such that for every $\mu \in (0,1/2]$ and $p \in [0,1]$,
$$
\E\Ll[\Ll|\na \phi\p_\mu - \na \td{\phi}\p_\mu \Rr|\Rr] \le c p^2 \mu^{-1+\beta} + e^{-\mu^{-\gamma}}.
$$
\end{thm}
\begin{rem}
We will see in section~\ref{s:error} that the upper bound can be improved using the recent results of \cite{marott}, see Theorem~\ref{t:sharp}.
\end{rem}
We start with several lemmas concerning the effect of perturbations on the corrector or the Green function.
\begin{lem}[perturbation at one edge]
\label{l:one}
Let
\begin{equation}
\label{defovphie}
\ov{\phi}^e_\mu = \phi^e_\mu - \phi\z_\mu.
\end{equation}
For every $\mu > 0$ and $x \in \Z^d$, we have
$$
\ov{\phi}^e_\mu(x) =- \na_2 G^e_\mu(x,\un{e}) \cdot C^e(\un{e}) (\xi + \na \phi\z_\mu(\un{e})).
$$
\end{lem}
\begin{proof}
In view of equation \eqref{defphimu} satisfied by $\phi_\mu\z$ and of the identity $A^e = A\z + C^e$, we infer that
$$
\mu \phi\z_\mu - \na^* \cdot A^e (\xi + \na \phi\z_\mu) = - \na^* \cdot C^e (\xi + \na \phi\z_\mu).
$$
This and equation \eqref{defphiE} satisfied by $\phi_\mu^e$ lead to
\begin{equation}
\label{propovphie}
\mu \ov{\phi}^e_\mu - \na^* \cdot A^e \na \ov{\phi}^e_\mu = \na^* \cdot C^e (\xi + \na \phi\z_\mu).
\end{equation}
The function on the right-hand side above is non-zero only at a finite number of points, so it is clearly of sub-exponential growth. The function $\ov{\phi}^e_\mu$ is also of sub-exponential growth by \eqref{defphiE-green} and \eqref{greenpoint1}. By Theorem~\ref{t:ex-un}, 
\begin{eqnarray*}
\ov{\phi}^e_\mu(x) & = & \sum_y G^e_\mu(x,y) \ \na^*\cdot C^e (\xi + \na \phi\z_\mu)(y) \\
& = & -\sum_y \na_2 G^e_\mu(x,y) \cdot C^e(y) (\xi + \na \phi\z_\mu(y)) \\
& = & -\na_2 G^e_\mu(x,\un{e}) \cdot C^e(\un{e}) (\xi + \na \phi\z_\mu(\un{e})),
\end{eqnarray*}
where in the second equality, the integration by parts is justified since only a finite number of terms in the sum are non-zero.
\end{proof}
\begin{rem}
\label{r:deftdphi}
As was recalled in the beginning of section~\ref{s:loc}, the gradient of $\phi_\mu\z$ is square-integrable: $\E[|\na \phi\z_\mu|^2] < \infty$. hence, by the ergodic theorem, 
\begin{equation}
\label{e:ergodic}
\frac{1}{|B_n|} \sum_{x \in B_n} |\na \phi_\mu\z(x)| \xrightarrow[n \to \infty]{\text{a.s.}} \E[|\na \phi_\mu\z|] < +\infty.
\end{equation}
It thus follows from Lemma~\ref{l:one} and the decay of the Green function \eqref{greenpoint1} that for every $x$, $\ov{\phi}_\mu^e(x)$ decays exponentially fast with the distance from $x$ to the edge $e$. In particular, for every $E \subset \B$, the function~$\td{\phi}_\mu^E$ defined in \eqref{deftdphi} is always well-defined and of sub-exponential growth.
\end{rem}

\begin{lem}[Perturbation at several edges]
\label{l:several}
Let $E \subset \B$, and let
\begin{equation}
\label{defovphip}
\ov{\phi}^E_\mu = \phi^E_\mu - \td{\phi}^E_\mu.
\end{equation}
For every $\mu > 0$ and $x \in \Z^d$, we have
$$
\ov{\phi}^E_\mu(x) = - \sum_{e \neq e' \in E} \na_2 G^E_\mu(x,\un{e}') \cdot C^{E\setminus \{e \}}(\un{e}') \na \ov{\phi}^e_\mu(\un{e}').
$$
\end{lem}
\begin{proof}
If $e \in E$, then using \eqref{propovphie} and the fact that $A^E = A^e + C^{E\setminus\{e\}}$, we have
$$
\mu \ov{\phi}^e_\mu - \na^* \cdot A^E \na \ov{\phi}^e_\mu = \na^* \cdot C^e (\xi + \na \phi\z_\mu) - \na^* \cdot C^{E\setminus \{e \}} \na \ov{\phi}^e_\mu.
$$
Similarly, we have
$$
\mu \phi\z_\mu - \na^* \cdot A^E (\xi + \na \phi\z_\mu) = - \na^* \cdot C^E (\xi + \na \phi\z_\mu).
$$
Recalling that
$$
\ov{\phi}^E_\mu = \phi^E_\mu - \phi\z_\mu - \sum_{e \in E} \ov{\phi}^e_\mu,
$$
that $C^E = \sum_{e \in E} C^e$, and equation \eqref{defphiE} satisfied by $\phi^E_\mu$, we obtain
$$
\mu \ov{\phi}^E_\mu - \na^* \cdot A^E \na \ov{\phi}^E_\mu = \sum_{e \in E} \na^* \cdot C^{E\setminus \{e \}} \na \ov{\phi}^e_\mu.
$$
By Remark~\ref{r:deftdphi}, the function $\ov{\phi}^E_\mu$ is of sub-exponential growth. Moreover, arguing as in this remark, we see that the function in the right-hand side above is also of sub-exponential growth. By Theorem~\ref{t:ex-un}, we thus have
\begin{eqnarray*}
\ov{\phi}^E_\mu(x) & = & \sum_{e \in E} \sum_{y \in \Z^d}  G^E_\mu(x,y) \ \na^* \cdot C^{E\setminus \{e \}} \na \ov{\phi}^e_\mu(y) \\
& = & - \sum_{e \in E} \sum_{y \in \Z^d} \na_2 G^E_\mu(x,y) \cdot C^{E\setminus \{e \}}(y) \na \ov{\phi}^e_\mu(y) \\
& = & - \sum_{e \neq e' \in E} \na_2 G^E_\mu(x,\un{e}') \cdot C^{E\setminus \{e \}}(\un{e}') \na \ov{\phi}^e_\mu(\un{e}'),
\end{eqnarray*}
where the integration by parts in the second equality is justified using the exponential decay of the Green function and the fact that $y \mapsto |\ov{\phi}^e_\mu(y)|$ is of sub-exponential growth, see Lemma~\ref{l:one} and \eqref{e:ergodic}.
\end{proof}
\begin{lem}[averaged control of the perturbed Green function]
For every $\zeta \in (0,1]$, there exists $c > 0$ such that for every $\mu \in (0,1/2]$ and $e,e' \in \B$,
$$
\E\Ll[ | \na_1 \na_2 G^e_\mu(\un{e}',\un{e})|^{1/\zeta}  \Rr] \le \frac{c}{(1\vee |\un{e}-\un{e}'|)^d}.
$$
\label{l:greenperturbed}
\end{lem}
\begin{proof}
As a first step, we show that the lemma is true if the perturbed Green function $G^e_\mu$ is replaced by the non-perturbed one $G\z_\mu$. 

We know from \eqref{greenaverage2} that
$$
\E\Ll[ | \na_1 \na_2 G\z_\mu(\un{e}',\un{e})| \Rr] \le \frac{c}{(1\vee |\un{e}-\un{e}'|)^d}.
$$
Moreover, by \eqref{greenpoint2}, the quantity $\na_2G^\o_\mu(x,y)$ is bounded uniformly over $\o \in \O$, $x,y \in \Z^d$ and $\mu > 0$. In particular, the random variable 
$$
| \na_1 \na_2 G\z_\mu(\un{e}',\un{e})|
$$
is bounded. Hence, the lemma is true if $G^e_\mu$ is replaced by $G\z_\mu$.


Let us write $\ov{G}^e_\mu = G^e_\mu - G\z_\mu$. In order to prove the lemma, it now suffices to argue that 
$$
\E\Ll[ | \na_1 \na_2 \ov{G}^e_\mu(\un{e}',\un{e})|^{1/\zeta}  \Rr]\le \frac{c}{(1\vee |\un{e}-\un{e}'|)^d}.
$$
Let $x \in \Z^d$. The Green function $G\z_\mu(x,\cdot)$ satisfies
$$
\mu G\z_\mu(x,\cdot) -\na^* \cdot A\z \na G\z_\mu(x,\cdot) = \1_x.
$$
The same equation holds for $G^e_\mu(x,\cdot)$ provided $A\z$ is replaced by $A^e = A\z + C^e$, hence
$$
\mu G^e_\mu(x,\cdot) -\na^* \cdot A\z \na G^e_\mu(x,\cdot) = \1_x + \na^* \cdot C^e \na G^e_\mu(x,\cdot).
$$
Combining the two previous identities, we obtain that
$$
\mu \ov{G}^e_\mu(x,\cdot) - \na^* \cdot A\z \na \ov{G}^e_\mu(x,\cdot) = \na^* \cdot C^e \na G^e_\mu(x,\cdot).
$$
By Theorem~\ref{t:ex-un}, we are led to
\begin{eqnarray*}
\ov{G}^e_\mu(x,y) & = & \sum_{z \in \Z^d} G\z_\mu(y,z) \ \na^* \cdot C^e \na G^e_\mu(x,\cdot)(z) \\
& = & - \sum_{z \in \Z^d} \na_2 G\z_\mu(y,z) \ \cdot C^e \na G^e_\mu(x,\cdot)(z) \\
& = & \na_2G\z_\mu(y,\un{e}) \ \cdot C^e(\un{e}) \na_2 G^e_\mu(x,\un{e}),
\end{eqnarray*}
and by symmetry of the Green functions, we also have 
\begin{equation}
\label{eqMO}
\ov{G}^e_\mu(x,y) = \na_2G\z_\mu(x,\un{e}) \ \cdot C^e(\un{e}) \na_2 G^e_\mu(y,\un{e}).
\end{equation}
It thus follows that
\begin{equation}
\label{relovG}
\na_1 \na_2\ov{G}^e_\mu(x,y) = C^e(\un{e}) \Ll(\na_1 \na_2 G^e_\mu(y,\un{e})\Rr) \ \Ll(\na_1 \na_2G\z_\mu(x,\un{e})\Rr),
\end{equation}
which we need to evaluate for $x = \un{e}'$ and $y = \un{e}$. Since, by \eqref{greenpoint2},
$$
\Ll|  C^e(\un{e}) \na_1 \na_2 G^e_\mu(\un{e},\un{e}) \Rr|
$$
is bounded uniformly over $\mu$ and $\un{\o}$, it suffices to see that
$$
\E\Ll[ |\na_1 \na_2G\z_\mu(\un{e}',\un{e})|^{1/\zeta} \Rr] \le \frac{c}{(1\vee |\un{e}-\un{e}'|)^d},
$$
but this was obtained during the first step of this proof, so we are done.
\end{proof}
\begin{proof}[Proof of Theorem~\ref{t:approxcorr}]
Note first that
$$
\E\Ll[\Ll|\na \phi\p_\mu - \na \td{\phi}\p_\mu \Rr|\Rr] = \E\Ll[\Ll|\na \ov{\phi}\p_\mu\Rr|\Rr],
$$
where as usual, we write $\ov{\phi}\p_\mu = \ov{\phi}^{E\p}_\mu$. By Lemma~\ref{l:several}, we have
\begin{equation*}
\na \ov{\phi}\p_\mu(0) = - \sum_{e \neq e' \in E\p} \Ll(\na_1 \na_2 G\p_\mu(0,\un{e}') \Rr) \ C^{E\p\setminus \{e \}}(\un{e}') \na \ov{\phi}^e_\mu(\un{e}'),
\end{equation*}
so we have the bound
\begin{equation*}
|\na \ov{\phi}\p_\mu(0)| \le c_+ \sum_{e \neq e' \in E\p} |\na_1 \na_2 G\p_\mu(0,\un{e}')| \ |\na \ov{\phi}^e_\mu(\un{e}')|.
\end{equation*}
By Lemma~\ref{l:one}, the right-hand side above is bounded, up to a constant, by
\begin{equation}
\label{e:gradovphip}
\sum_{e \neq e' \in E\p}| \na_1 \na_2 G\p_\mu(0,\un{e}')| \ | \na_1 \na_2 G^e_\mu(\un{e}',\un{e})| \ |\xi + \na \phi\z_\mu(\un{e})|.
\end{equation}

Let $\gamma > 1/2$, and write $E\p_\mu = E\p \cap B_{\mu^{-\gamma}}$, where we recall that $B_n = \{-n, \ldots, n\}^d$. We will argue that in the above sum, the only terms which truly contribute to the sum are those for which $e,e' \in E\p_\mu$. In order to do so, it is convenient to introduce $\td{\gamma}$ such that $1/2 < \td{\gamma} < \gamma$, and to define $\td{E}\p_\mu = E\p \cap B_{\mu^{-\td{\gamma}}}$.

Recall that $\E[|\na \phi\z_\mu(\un{e})|] \le \E[|\na \phi\z_\mu(\un{e})|^2]^{1/2}$ does not depend on $e$ and is bounded uniformly over $\mu > 0$. We also have from \eqref{greenpoint2} that for every $E \subset \B$,
\begin{equation}
\label{greepointsimpl}
|\na_2 G^E_\mu(x,y)| \le c \ e^{-\td{c} \sqrt{\mu} |y-x|},
\end{equation}
with $\td{c} > 0$. Hence, up to a constant, the expectation
\begin{equation}
\label{e:gradovphip2}
\E\Ll[ \sum_{e' \in E\p \setminus \td{E}\p_\mu} \ \sum_{e \in E\p, e \neq e'}  | \na_1 \na_2 G\p_\mu(0,\un{e}')| \ | \na_1 \na_2 G^e_\mu(\un{e}',\un{e})| \ |\xi + \na \phi\z_\mu(\un{e})| \Rr]
\end{equation}
is bounded by
$$
\sum_{z' \in \Z^d \setminus B_{\mu^{-\td{\gamma}}}} \ \sum_{z \in \Z^d} e^{-\td{c}\sqrt{\mu} |z'|} \ e^{-\td{c}\sqrt{\mu} |z-z'|}.
$$
Summation over $z$ gives a contribution of order $\mu^{-d/2}$, while the remaining summation over $z'$ decays as $e^{-\mu^{-\gamma'}}$ for every $\gamma' < \td{\gamma}-1/2$. We have thus justified that the difference between the expectation of \eqref{e:gradovphip} and \eqref{e:gradovphip2} is $O(e^{-\mu^{-\gamma'}})$ for every $\gamma' < \td{\gamma}-1/2$. We can thus focus on studying
$$
\E\Ll[ \sum_{e' \in  \td{E}\p_\mu, e \in E\p, e \neq e'}  | \na_1 \na_2 G\p_\mu(0,\un{e}')| \ | \na_1 \na_2 G^e_\mu(\un{e}',\un{e})| \ |\xi + \na \phi\z_\mu(\un{e})| \Rr].
$$
A similar reasoning enables to show that in the above sum, the terms for which $e \notin E\p_\mu$ can be discarded. Indeed, the contribution of those terms can be bounded, up to a constant, by
\begin{equation}
\label{eq1234}
\sum_{z' \in B_{\mu^{-\td{\gamma}}}} \ \sum_{z \in \Z^d \setminus B_{\mu^{-\gamma}}} e^{-\td{c} \sqrt{\mu} |z'|} \ e^{-\td{c} \sqrt{\mu} |z-z'|}.
\end{equation}
Since $\td{\gamma} < \gamma$, the distance between $z$ and $z'$ in any term of this double sum is always greater than $\mu^{-\gamma}/2$. Hence, we obtain an upper bound for the inner sum in \eqref{eq1234} if we sum over all $z$ such that $|z-z'| \ge \mu^{-\gamma}/2$. This gives a sum of order $O(e^{-\mu^{-\gamma'}})$ for some $\gamma' > 0$, and hence a similar bound holds for the total sum in \eqref{eq1234}. We have thus argued that it suffices to study
$$
\E\Ll[ \sum_{e' \in  \td{E}\p_\mu, e \in E\p_\mu, e \neq e'}  | \na_1 \na_2 G\p_\mu(0,\un{e}')| \ | \na_1 \na_2 G^e_\mu(\un{e}',\un{e})| \ |\xi + \na \phi\z_\mu(\un{e})| \Rr].
$$
Finally (in order to keep light notations), arguing as in the first step, we see that we can as well consider the sum as ranging over all $e, e' \in E\p_\mu$, $e \neq e'$. 

We now decompose the expectation into
\begin{multline}
\label{e:thestep}
\E\Ll[ \sum_{e\neq e' \in  E\p_\mu}  | \na_1 \na_2 G\p_\mu(0,\un{e}')| \ | \na_1 \na_2 G^e_\mu(\un{e}',\un{e})| \ |\xi + \na \phi\z_\mu(\un{e})| \Rr] \\
= \sum_E \sum_{e\neq e' \in E} \mcl{E}(e,e',E) \ \P[E\p_\mu = E],
\end{multline}
where in the first sum, $E$ ranges over all possible subsets of $B_{\mu^{-\gamma}}$, and where we write
$$
\mcl{E}(e,e',E) = \E\Ll[ | \na_1 \na_2 G\p_\mu(0,\un{e}')| \ | \na_1 \na_2 G^e_\mu(\un{e}',\un{e})| \ |\xi + \na \phi\z_\mu(\un{e})| \ \Big| \ E\p_\mu = E\Rr].
$$
For any $\zeta\in (0,1)$, we have by H\"older's inequality
\begin{multline}
\label{e:mcle}
\mcl{E}(e,e',E) \le \E\Ll[ \Ll(| \na_1 \na_2 G\p_\mu(0,\un{e}')| \ | \na_1 \na_2 G^e_\mu(\un{e}',\un{e})|\Rr)^{1/\zeta} \ \Big| \ E\p_\mu = E\Rr]^\zeta \\
\E\Ll[ |\xi + \na \phi\z_\mu(\un{e})|^{1/(1-\zeta)} \ \Big| \ E\p_\mu = E\Rr]^{1-\zeta}.
\end{multline}
Since $|\xi + \na \phi\z_\mu(\un{e})|$ does not depend on $E\p$, the last expectation is simply
$$
\E\Ll[ |\xi + \na \phi\z_\mu(\un{e})|^{1/(1-\zeta)} \Rr]^{1-\zeta},
$$
which is bounded by a power of $\log(\mu^{-1})$ by Theorem~\ref{t:corr-moments}. 

For the first expectation in the right-hand side of \eqref{e:mcle}, we use the pointwise bound \eqref{greenpoint2} on the gradient of the Green function to get
\begin{equation*}
\begin{split}
& \E\Ll[ \Ll(| \na_1 \na_2 G\p_\mu(0,\un{e}')| \ | \na_1 \na_2 G^e_\mu(\un{e}',\un{e})|\Rr)^{1/\zeta} \ \Big| \ E\p_\mu = E\Rr]^\zeta \\
& \qquad \le \frac{c}{(1\vee |\un{e}'|)^{d-2+\alpha}} \ \E\Ll[ | \na_1 \na_2 G^e_\mu(\un{e}',\un{e})|^{1/\zeta} \ \Big| \ E\p_\mu = E\Rr]^\zeta \\
& \qquad = \frac{c}{(1\vee |\un{e}'|)^{d-2+\alpha}} \ \E\Ll[ | \na_1 \na_2 G^e_\mu(\un{e}',\un{e})|^{1/\zeta} \Rr]^\zeta,
\end{split}
\end{equation*}
and Lemma~\ref{l:greenperturbed} ensures that the latter is bounded by
$$
\frac{c}{(1\vee |\un{e}'|)^{d-2+\alpha} \ (1\vee |\un{e}-\un{e}'|)^{\zeta d}}.
$$
It thus follows that the right-hand side of \eqref{e:thestep} is bounded by some power of $\log(\mu^{-1})$ times
\begin{multline*}
\E\Ll[ \sum_{e \neq e' \in E\p_\mu}\frac{1}{(1\vee |\un{e}'|)^{d-2+\alpha} \ (1\vee |\un{e}-\un{e}'|)^{\zeta d}} \Rr] \\
= p^2 \sum_{e \neq e' \in B_{\mu^{-\gamma}} }\frac{1}{(1\vee |\un{e}'|)^{d-2+\alpha} \ (1\vee |\un{e}-\un{e}'|)^{\zeta d}}.
\end{multline*}
To sum up, we have shown that there exists $\gamma' > 0$ such that for every $\zeta < 1$,
$$
\E\Ll[\Ll|\na \ov{\phi}\p_\mu\Rr|\Rr] \le c  p^2 \log^r(\mu^{-1}) \sum_{e \neq e' \in B_{\mu^{-\gamma}} }\frac{1}{(1\vee |\un{e}'|)^{d-2+\alpha} \ (1\vee |\un{e}-\un{e}'|)^{\zeta d}} + ce^{-\mu^{-\gamma'}},
$$
for some exponent $r \ge 0$.
Up to a multiplicative constant, the sum over $e,e'$ above can be bounded by the integral
\begin{eqnarray*}
\int_{|x|, |y| \le \mu^{-\gamma}} \frac{\d x \ \d y}{|x|^{d-2+\alpha}|y-x|^{\zeta d}} &  \le & \int_{|x| \le \mu^{-\gamma}, |y-x| \le 2\mu^{-\gamma}}  \frac{\d x \ \d y}{|x|^{d-2+\alpha}|y-x|^{\zeta d}} \\
& \le & c \mu^{-\gamma(d-\zeta d +2-\alpha)}.
\end{eqnarray*}
Until now, the parameters $\gamma > 1/2$ and $\zeta < 1$ were arbitrary. We can choose them in such a way that
$$
\gamma(d-\zeta d +2-\alpha) < 1,
$$
and this finishes the proof of Theorem~\ref{t:approxcorr}.
\end{proof}

%

%
%
%
%
%
%
%
%
\section{Solving stationary elliptic equations}
\label{s:station}
Recall that we say that a function $\psi : \Z^d \times \un{\Omega} \to \R$ is stationary if $\psi(x,\un{\omega}) = \psi(0,\theta_x \ \un{\omega})$, where $(\theta_x)_{x \in \Z^d}$ denotes the action of $\Z^d$ on $\un{\Omega}$. The stationary function $\psi$ is thus fully characterized by the knowledge of the function $\un{\omega} \to \psi(0,\un{\omega})$, with which it can be identified.

The first purpose of this section is to recall existence and uniqueness results for elliptic equations. Contrary to what was done in section~\ref{s:solve}, the equations considered here do not contain a zero-order regularizing parameter. On the other hand, we will work with the extra assumption of stationarity. Because of the identification mentioned above, we will focus on functions defined on $\un{\Omega}$ (instead of $\Z^d \times \un{\Omega}$). We may keep this identification implicit, writing for instance $A\z$ instead of $A\z(0,\cdot)$.

For $f : \un{\Omega} \to \R$, we define the forward gradient $D f : \un{\Omega} \to \R^d$ as  
$$
D f(\un{\o})=\left[  
\begin{array}{c}
f(\theta_{\e_1} \ \un{\o})-f(\un{\o}) \\
\vdots\\
f(\theta_{\e_d} \ \un{\o})-f(\un{\o})
\end{array}
\right],
$$
and for $F = (F_1,\ldots,F_d): \un{\O} \to \R^d$, we write $D^* \cdot F$ for the backward divergence, 
$$
D^* \cdot F(\un{\o}) = \sum_{i = 1}^d \Ll(F_i(\un{\o}) - F_i(\theta_{-\e_i} \ \un{\o})\Rr).
$$
We write $\|F\|_2$ for the $L^2$ norm of $F$, that is,
$\|F\|_2^2 = \sum_{i = 1}^d \E[F_i^2]$. We let $L^2(\un{\O})$ be the space of functions $F : \un{\O} \to \R^d$ such that $\|F\|_2$ is finite, and $L^2_\na$ be the closure in $L^2(\un{\O})$ of $\{D f, \ f : \un{\O} \to \R \text{ s.t. } \E[f^2] < \infty \}$. 
\begin{thm}[Existence and uniqueness of solutions]
\label{t:station}
	For every $F \in L^2(\un{\O})$, there exists a unique $\chi \in L^2_\na$ such that
\begin{equation}
\label{e:station}
- D^* \cdot A\z \chi =  D^* \cdot F.
\end{equation}
\end{thm}
\begin{proof}
We begin by showing the existence of solutions, following roughly the arguments of \cite{kipvar}. (Although the proof is by now standard, it is useful to recall its workings since we will need to extend it slightly later on.) The operator $\L := -D^* \cdot A\z D$ is self-adjoint and positive on $\{f : \un{\O} \to \R : \E[f^2] < \infty\}$. The positivity comes from the observation that, for every square-integrable $f$,
\begin{equation}
\label{posit}
\E[f \ \L f] = \E[Df \cdot A\z Df]
.
\end{equation}
Hence, for every $\mu > 0$, there exists a square-integrable $\Psi_\mu : \un{\O} \to \R$ such that
\begin{equation}
\label{e:Psi}
(\mu + \L)\Psi_\mu = D^* \cdot F.
\end{equation} 
Let $f  = D^* \cdot F$. Since $f$ is square-integrable, by the spectral theorem, there exists a measure $e_f$ on $\R_+$ such that for every bounded continuous function $G : \R_+ \to \R$, one has
\begin{equation}
\label{e:spectr}
\E\Ll[ f \ G(\L)(f) \Rr] = \int G(\lambda) \ \d e_f(\lambda).
\end{equation}
Note that for every square-integrable function $g$, one has
$$
\E\Ll[f \ g\Rr] = - \E[F \cdot Dg] \le \|F\|_2 \ \|Dg\|_2,
$$
by the Cauchy-Schwarz inequality. Furthermore, by \eqref{posit},
$$
\|Dg\|_2^2 \le \frac{1}{c_-} \E[g \ \L g],
$$
so that 
$$
\E\Ll[f \ g\Rr] \le \frac{\|F\|_2}{\sqrt{c_-}} \ \E[g \ \L g]^{1/2}.
$$
For every $n \in \N$, let $G_n(\lambda) = \lambda^{-1} \wedge n$. Taking $g = G_n(\L)(f)$ in the inequality above, and using \eqref{e:spectr}, we get
$$
\int (\lambda^{-1} \wedge n) \ \d e_f(\lambda) \le \frac{\|F\|_2}{\sqrt{c_-}} \Ll(\int \lambda(\lambda^{-1} \wedge n)^2 \ \d e_f(\lambda)\Rr)^{1/2}.
$$
By the monotone convergence theorem, this turns into
\begin{equation}
\label{estimH-1}
\int \frac{1}{\lambda} \ \d e_f(\lambda) \le \frac{\|F\|_2^2}{c_-}.
\end{equation}
Equipped with this inequality, we can show that $(D\Psi_\mu)_{\mu > 0}$ is a Cauchy sequence in $L^2(\un{\O})$. Indeed, note that by \eqref{posit}, we have
\begin{equation}
\label{e:stepcauchy}
\|D\Psi_\mu - D\Psi_\nu\|_2^2 \le \frac{1}{c_-} \E[(\Psi_\mu - \Psi_\nu) \ \L(\Psi_\mu - \Psi_\nu)].
\end{equation}
Since $\Psi_\mu = (\mu + \L)^{-1} f$, and using \eqref{e:spectr}, we can rewrite the expectation in the right-hand side above as
$$
\int \lambda \Ll( (\mu + \lambda)^{-1} - (\nu + \lambda)^{-1} \Rr)^2 \ \d e_f(\lambda) = \int \frac{\lambda (\nu - \mu)^2}{(\mu + \lambda)^2 (\nu + \lambda)^2} \ \d e_f(\lambda).
$$
Without loss of generality, we may assume that $\mu < \nu$. In this case, the integrand above is bounded by
$$
\frac{\lambda \nu^2}{\lambda^2 \nu^2} = \frac{1}{\lambda},
$$
which is integrable by \eqref{estimH-1}. Now, for every fixed $\lambda$, we have
$$
\frac{\lambda (\nu - \mu)^2}{(\mu + \lambda)^2 (\nu + \lambda)^2} \le \frac{\lambda \nu^2}{\lambda^4},
$$
which tends to $0$ as $\mu,\nu \to 0$. By the dominated convergence theorem, we thus obtain that the left-hand side of \eqref{e:stepcauchy} tends to $0$ as $\mu,\nu \to 0$, and thus that $(D\Psi_\mu)_{\mu > 0}$ is a Cauchy sequence in $L^2(\un{\O})$. Let us write $\chi$ for the limit. By definition, we have $\chi \in L^2_\na$. 

We now show that $\chi$ satisfies \eqref{e:station}. One can check that, as a consequence of the estimate \eqref{estimH-1} (and using \eqref{e:spectr}),
\begin{equation}
\label{estim2}
\sqrt{\mu} \Psi_\mu \xrightarrow[\mu \to 0]{L^2(\un{\O})} 0.
\end{equation}
For every square-integrable $g : \un{\O} \to \R$, we can write the weak formulation of \eqref{e:Psi}:
$$
\mu \E[g \ \Psi_\mu] + \E[D g \cdot A\z D\Psi_\mu] = \E[g \ D^* \cdot F].
$$
Using \eqref{estim2} and the Cauchy-Schwarz inequality, we get that $\mu \E[g \ \Psi_\mu]$ tends to $0$ as $\mu$ tends to $0$, and thus
\begin{equation}
\label{e:station-weak}
\E[D g \cdot A\z \chi] = \E[g \ D^* \cdot F].
\end{equation}
The left-hand side above is equal to $-\E[g \ D^* \cdot A\z \chi]$. Hence, the function $D^* \cdot A\z \chi + D^* \cdot F$ is orthogonal to every function in $L^2(\un{\O})$. It is thus equal to zero, and that is to say that $\chi$ satisfies \eqref{e:station}.

We now turn to uniqueness. By linearity, it suffices to show uniqueness for $F = 0$. Let $\chi \in L^2_\na$ satisfy 
\begin{equation}
\label{e:station0}
-D^* \cdot A\z \chi = 0.
\end{equation}
There exists a sequence of square-integrable $f_n : \un{\O} \to \R$ such that $D f_n$ converges to~$\chi$ in $L^2(\un{\O})$. By the weak formulation of \eqref{e:station0}, we have
$$
\E[Df_n \cdot A\z \chi] = 0.
$$
Passing to the limit, we get $\E[\chi \cdot A\z \chi] = 0$, which implies that $\chi = 0$, and thus finishes the proof.
\end{proof}

We now proceed to derive several convergence results that will be useful for our subsequent reasoning.

\begin{prop}[convergence of correctors]
\label{p:corrlim}
The function $\na \phi\z_\mu(0,\cdot)$ converges in $L^2(\un{\O})$, as $\mu$ tends to $0$, to the function $\na \phi\z(0,\cdot)$, while for every $e \in \B$, the function
$\na \phi^e_\mu(0,\cdot)$ converges in $L^2(\un{\O})$, as $\mu$ tends to $0$, to a function that we write $\na \phi^e(0,\cdot)$.
\end{prop}
\begin{proof}
The first part is classical (and is a minor adaptation of the proof of Theorem~\ref{t:station}, choosing $F = A\z \xi$). For the second part, in view of Lemma~\ref{l:one}, it suffices to study the convergence of $\na_2 G^\o_\mu(x,y)$, for any fixed $x, y \in \Z^d$ and $\omega \in \Omega$. We define
$$
\na_2 G^\o(x,y) = \int_0^{+\infty} \na_3 q^\o(t,x,y) \ \d t.
$$
Note that this definition makes sense even in dimension 2 by \eqref{parabpointgrad}, although $G^\o(x,y)$ itself does not. For the same reason, $\na_2 G^\o(x,y)$ is bounded uniformly over $\omega \in \Omega$. We get that
$$
\na_2 G^\o(x,y) - \na_2 G^\o_\mu(x,y) = \int_0^{+\infty} (1-e^{-\mu t}) \na_3 q^\o(t,x,y) \ \d t.
$$
This and \eqref{parabpointgrad} ensure that $\na_2 G^\o_\mu(x,y)$ converges to $\na_2 G^\o(x,y)$ as $\mu$ tends to $0$, uniformly over $\omega \in \Omega$. Using Lemma~\ref{l:one}, we thus obtain the convergence of $\na \phi^e_\mu$ in $L^2(\un{\O})$, as desired.
\end{proof}
\begin{rem}
\label{r:phie}
The proof shows that
\begin{equation}
\label{e:naphie}
\na \phi^e(0,\un{\o}) = \na \phi\z(0,\un{\o}) - \Ll( \na_1 \na_2 G^e(0,\un{e})\Rr)  C^e(\un{e}) (\xi + \na \phi\z(\un{e}) ).
\end{equation}
\end{rem}
\begin{prop}[convergence of full correction]
\label{p:defcoef}
Let $\psi_\mu : \un{\O} \to \R$ be defined by
$$
\psi_\mu = \sum_{e \in \B} (\phi^e_\mu - \phi\z_\mu)(0,\cdot).
$$
The gradient $D \psi_\mu$ converges in $L^2(\un{\O})$, as $\mu$ tends to $0$, to $\chi \in L^2_\na$ the unique solution of
\begin{equation}
\label{e:defcoef}
-D^* \cdot A\z \chi = D^* \cdot F,
\end{equation}
where $F \in L^2(\un{\O})$ is defined by
\begin{equation}
\label{defF}
F = \sum_{e \in \B} C^e (\xi + \na \phi^e)(0,\cdot).
\end{equation}
\end{prop}
\begin{rem}
Note that $\psi_\mu$ is well-defined, since $\td{\phi}^\B_\mu$ introduced in \eqref{deftdphi} is. Also, in the definition of $F$, note that the presence of the term $C^e$ ensures that every summand indexed by an edge $e$ such that $\un{e} \neq 0$ is actually equal to $0$. 
\end{rem}
\begin{proof}
Recall that, by definition, $\phi^e_\mu$ satisfies
$$
\mu \phi^e_\mu - \na^* \cdot A^e (\xi + \na \phi^e_\mu) = 0 \qquad (\text{in } \Z^d),
$$
and that $A^e = A\z+C^e$. As a consequence, we get
$$
\mu \phi^e_\mu - \na^* \cdot A\z (\xi + \na \phi^e_\mu) = \na^* \cdot C^e (\xi + \na \phi^e_\mu) \qquad (\text{in } \Z^d).
$$
Using also the definition of $\phi\z_\mu$, see \eqref{defphimu}, and recalling that we write $\ov{\phi}^e_\mu = \phi^e_\mu - \phi\z_\mu$, we obtain that
$$
\mu \ov{\phi}^e_\mu - \na^* \cdot A\z  \na \ov{\phi}^e_\mu = \na^* \cdot C^e (\xi + \na \phi^e_\mu) \qquad (\text{in } \Z^d),
$$
and thus
\begin{equation}
\label{e:defcoef1}
\mu \psi_\mu - \na^* \cdot A\z  \na \psi_\mu = \sum_{e \in \B} \na^* \cdot C^e (\xi + \na \phi^e_\mu) \qquad (\text{in } \Z^d).
\end{equation}
We let 
\begin{equation}
\label{defFmu}
F_\mu(x,\un{\o}) = \sum_{e \in \B} C^e (\xi + \na \phi^e_\mu)(x,\un{\o}).
\end{equation}
Note that $F_\mu$ is a stationary function, so we can identify it with the function
$$
\left\{
\begin{array}{lll}
\un{\O} & \to & \R^d \\
\un{\o} & \mapsto & F_\mu(0,\un{\o}).
\end{array}
\right.
$$
With this identification in mind, we can rewrite \eqref{e:defcoef1} as
\begin{equation}
\label{e:eqFmu}
\Ll( \mu - D^* \cdot A\z D \Rr) \psi_\mu = D^* \cdot F_\mu.
\end{equation}
By Proposition~\ref{p:corrlim}, the function $F_\mu$ converges in $L^2(\un{\O})$ to $F$ defined in \eqref{defF}.
We would like to argue as in the ``existence'' part of the proof of Theorem~\ref{t:station} to show that $D\psi_\mu$ converges to $\chi$ satisfying \eqref{e:defcoef}, but a difficulty arises since the right-hand side of \eqref{e:eqFmu} now depends on $\mu$. Let $\td{\psi}_\mu \in L^2(\un{\O})$ be the unique solution to
$$
\Ll( \mu - D^* \cdot A\z D \Rr) \td{\psi}_\mu = D^* \cdot F.
$$
By Theorem~\ref{t:station}, we know that $D\td{\psi}_\mu$ converges in $L^2(\un{\O})$ to $\chi \in L^2_\na$ the unique solution of \eqref{e:defcoef}. We now let $\ov{\psi}_\mu = \psi_\mu - \td{\psi}_\mu$, so that
$$
\Ll( \mu - D^* \cdot A\z D \Rr) \ov{\psi}_\mu = D^* \cdot (F_\mu - F).
$$
In order to conclude, we need to show that $D\ov{\psi}_\mu$ tends to $0$ in $L^2(\un{\O})$. Let us write $f_\mu = D^* \cdot (F_\mu - F)$. In view of \eqref{posit} and \eqref{e:spectr}, we have
$$
\|D\ov{\psi}_\mu\|_2^2 \le \frac{1}{c_-} \E[\ov{\psi}_\mu \ \L \ov{\psi}_\mu] = \frac{1}{c_-} \int\frac{\lambda}{(\mu + \lambda)^2} \ \d e_{f_\mu}(\lambda) \le \frac{1}{c_-} \int \frac{1}{\lambda} \ \d e_{f_\mu}(\lambda).
$$
By \eqref{estimH-1}, the last integral is bounded by
$$
\frac{\|F_\mu - F\|_2^2}{c_-},
$$
and since $F_\mu$ tends to $F$ in $L^2(\un{\O})$, this finishes the proof.
\end{proof}

%
%
%
%
%
%
%
%
\section{First-order expansion around $p = 0$}
\label{s:expansion0}

We are now ready to state and prove the first-order expansion around $p = 0$ of the homogenized matrix $\Ah\p$.

\begin{thm}[first-order expansion around $p=0$]
\label{t:main0}
There exists $a_1\z \in \R$ such that, as $p$ tends to $0$,
\begin{equation}
\label{e:main0}
\xi \cdot \Ah\p \xi = \xi \cdot \Ah\z \xi + p \ a_1\z + o(p).
\end{equation}
Moreover, the coefficient $a_1\z$ can be defined as the limit of $a_1\z(\mu)$ as $\mu$ tends to $0$, where $a_1\z(\mu)$ is given by
\begin{equation}
\label{defa1mu}
a_1\z(\mu) = \sum_{e \in \B} \Ll(\E[\xi \cdot A^e(\xi + \na {\phi}_\mu^e)] - \E[\xi \cdot A\z(\xi + \na {\phi}_\mu\z)] \Rr),
\end{equation}
and where we recall that $\phi_\mu^e$ satisfies \eqref{defphiE} with $E = \{e\}$, that $\phi_\mu\z$ satisfies \eqref{defphimu} with $p = 0$, and that the functions under the expectations in \eqref{defa1mu} are understood to be evaluated at $(0,\un{\o})$. Alternative characterizations of $a_1\z$ are given in Propositions~\ref{p:periodiz} and \ref{p:spatial-average} below.
\end{thm}
\begin{rem}
	The proof actually reveals that there exists an exponent $\ov{\eta} > 0$ depending on the ellipticity constants such that \eqref{e:main0} holds with $o(p)$ replaced by $o(p^{1+\ov{\eta}})$. We show in Theorem~\ref{t:error} that the error term is actually $o(p^{2-\eta})$ for every $\eta > 0$.
\end{rem}
\begin{proof}
For this proof, we fix
\begin{equation}
\label{fixmu}
\mu = p^{1+\eps},
\end{equation}
for some $\eps > 0$. Note that in order to avoid heavy notations, this dependence between $\mu$ and $p$ is kept implicit. For instance, and in view of Theorem~\ref{t:approx}, we may write without further comment that
\begin{equation}
\label{e:approx}
\Ll|\xi \cdot \Ah\p \xi - \E[\xi \cdot A\p(\xi + \na \phi_\mu\p)] \Rr|	= o\Ll(p\Rr) \qquad (p \to 0).
\end{equation}
For $\beta > 0$ given by Theorem~\ref{t:approxcorr}, we have
\begin{equation}
\label{e:approx2}
\Ll| \E[\xi \cdot A\p(\xi + \na \phi_\mu\p)]  - \E[\xi \cdot A\p(\xi + \na \td{\phi}_\mu\p)] \Rr| = O\Ll(p^{2-(1+\eps)(1-\beta)}\Rr)
.
\end{equation}
We fix the value of $\eps > 0$ in such a way that 
\begin{equation}
\label{e:fixeps}
2-(1+\eps)(1-\beta) > 1.
\end{equation}
Using the definition of $\td{\phi}_\mu$ introduced in \eqref{deftdphi} (with $E = E\p$), we can write
\begin{equation}
\label{e:1}
\E[\xi \cdot A\p(\xi + \na \td{\phi}_\mu\p)] = \E[\xi \cdot A\p(\xi + \na {\phi}_\mu\z)] + \E\Ll[ \sum_{e \in E\p} \xi \cdot A\p(\na \phi^e_\mu - \na \phi\z_\mu) \Rr].
\end{equation}
We analyse the two terms on the right-hand side separately, beginning with the first. Note that $A\p$ (evaluated at the origin) is equal to $A\z$ unless one of the edges attached to the origin belongs to $E\p$. Moreover, the event that two or more of these edges belong to $E\p$ has probability $O(p^2)$. As a consequence, and since $\E[|\na \phi_\mu\z|^2]$ is bounded uniformly over $\mu$, we have
\begin{multline}
\label{e:part1}
\E[\xi \cdot A\p(\xi + \na {\phi}_\mu\z)] \\
= \E[\xi \cdot A\z(\xi + \na {\phi}_\mu\z)] + p \sum_{e : \un{e} = 0} \Ll(\E[\xi \cdot A^e(\xi + \na {\phi}_\mu\z)] - \E[\xi \cdot A\z(\xi + \na {\phi}_\mu\z)] \Rr) + O(p^2).
\end{multline}
We now turn to the second term in the right-hand side of \eqref{e:1}, which we decompose into
\begin{equation}
\label{e:2}
\E\Ll[ \sum_{e \in E\p,  \un{e} = 0} \xi \cdot A\p(\na \phi^e_\mu - \na \phi\z_\mu) \Rr] + \E\Ll[ \sum_{e \in E\p,  \un{e} \neq 0} \xi \cdot A\p(\na \phi^e_\mu - \na \phi\z_\mu) \Rr].
\end{equation}
By Lemma~\ref{l:one} and \eqref{greenpoint2}, it is clear that $\E[|\na \phi^e_\mu|^2]$ is bounded uniformly over $\mu$. Hence, reasoning as above, we get
\begin{equation}
\label{e:3}
\E\Ll[ \sum_{e \in E\p,  \un{e} = 0} \xi \cdot A\p(\na \phi^e_\mu - \na \phi\z_\mu) \Rr] = p \sum_{e: \un{e} = 0} \E\Ll[ \xi \cdot A^e(\na \phi^e_\mu - \na \phi\z_\mu) \Rr] + O(p^2).
\end{equation}
We now consider the second term in \eqref{e:2}. For an edge $e$ such that $\un{e} \neq 0$, the event $e \in E\p$ is independent of the value of $A\p$ (at the origin). Hence, we can rewrite the second term in \eqref{e:2} as
$$
p \sum_{e : \un{e} \neq 0} \E[\xi \cdot A\p(\na \phi^e_\mu - \na \phi\z_\mu)],
$$
which itself can be rewritten as
$$
p \E[\xi \cdot A\p D\psi_\mu] - p \sum_{e : \un{e} = 0}\E[\xi \cdot A\p(\na \phi^e_\mu - \na \phi\z_\mu)],
$$
where $\psi_\mu$ was introduced in Proposition~\ref{p:defcoef}. Arguing as above, we see that, on the one hand,
$$
\sum_{e : \un{e} = 0}\E[\xi \cdot A\p(\na \phi^e_\mu - \na \phi\z_\mu)] = \sum_{e : \un{e} = 0}\E[\xi \cdot A\z(\na \phi^e_\mu - \na \phi\z_\mu)] + O(p),
$$
while on the other hand, since $D\psi_\mu$ remains bounded in $L^2(\un{\O})$ by Proposition~\ref{p:defcoef},
$$
\E[\xi \cdot A\p D\psi_\mu] = \E[\xi \cdot A\z D\psi_\mu] + O(p).
$$
We have thus shown that
\begin{multline*}
\E\Ll[ \sum_{e \in E\p,  \un{e} \neq 0} \xi \cdot A\p(\na \phi^e_\mu - \na \phi\z_\mu) \Rr] \\
= p\E[\xi \cdot A\z D\psi_\mu] - p\sum_{e : \un{e} = 0}\E[\xi \cdot A\z(\na \phi^e_\mu - \na \phi\z_\mu)] + O(p^2).
\end{multline*}
Combining this with \eqref{e:approx}, \eqref{e:approx2}, \eqref{e:1}, \eqref{e:part1}, \eqref{e:2} and \eqref{e:3}, we thus obtain
\begin{equation}
\label{e:presque}
\xi \cdot \Ah\p \xi = \E[\xi \cdot A\z(\xi + \na {\phi}_\mu\z)] + p \ a_1\z(\mu) + o(p),
\end{equation}
where we introduced
\begin{multline}
\label{defa1zmu}
a_1\z(\mu)=  \sum_{e : \un{e} = 0} \Ll(\E[\xi \cdot A^e(\xi + \na {\phi}_\mu^e)] - \E[\xi \cdot A\z(\xi + \na {\phi}_\mu\z)] \Rr) + \\
\E[\xi \cdot A\z D\psi_\mu] - \sum_{e : \un{e} = 0}\E[\xi \cdot A\z(\na \phi^e_\mu - \na \phi\z_\mu)].
\end{multline}
Using the definition of $\psi_\mu$, one can observe that this definition coincides with that given in \eqref{defa1mu}. Using Theorem~\ref{t:approx} again, and recalling \eqref{fixmu}, we see that $\E[\xi \cdot A\z(\xi + \na {\phi}_\mu\z)] = \xi \cdot \Ah\z \xi + o(p)$. In order to conclude the proof, it thus suffices to show that $a_1\z(\mu)$ converges to a constant as $\mu$ tends to $0$. But this is a consequence of Propositions~\ref{p:corrlim} and \ref{p:defcoef}.
\end{proof}

%
%
%
%
%
%
%
%
\section{Approximation by periodization}
\label{s:periodiz}
The aim of this section is to give alternative characterizations of $a_1\z$, based on computing correctors on periodizations of the medium. We recall that we write $B_n = \{-n, \ldots, n\}^d$, and we let
$\B_n = \{ e \in \B : \un{e} \in B_n\}$. 
We define the periodized environment $\on \in \un{\O}$ by letting, for every $e \in \B$,
$(\on)_e = \omega_{e+x}$, where $x$ is the unique element of $(2n+1)\Z^d$ such that $e+x \in \B_n$. By the Lax-Milgram lemma and the Poincaré inequality (or by basic linear-algebra considerations), for every $\un{\o} \in \un{\O}$, there exists a unique $B_n$-periodic function $\phi\z_n(\cdot, \on)$ with zero average over $B_n$ such that for every $x \in \Z^d$,
$$
-\na^* \cdot A\z (\xi + \na \phi\z_n)(x,\on) = 0.
$$
For the same reason, for every $e \in \B_n$, there exists a unique $B_n$-periodic function $\phi_n^e(\cdot,\on)$ with zero average over $B_n$ such that for every $x \in \Z^d$,
$$
-\na^* \cdot A^e (\xi + \na \phi^e_n)(x,\on) = 0.
$$
We write $\E_n[f]$ as a shorthand for $\E[f(\on)]$, whenever this is well-defined. As usual, if the function is of the form $f(x,\on)$, $x \in \Z^d$, we interpret $\E_n[f]$ to mean $\E[f(0,\on)]$.

\begin{prop}[approximation of $a_1\z$ by periodization]
\label{p:periodiz}
The coefficient $a_1\z$ appearing in Theorem~\ref{t:main0} is the limit of $a_1\z(n)$ as $n$ tends to infinity, where
\begin{equation}
\label{defa1n}
a_1\z(n) = \sum_{e \in \B_n} \Ll( \E_n[\xi \cdot A^e (\xi + \na \phi^e_n)]  - \E_n[\xi \cdot A\z (\xi + \na \phi\z_n)] \Rr).
\end{equation}
\end{prop}
\begin{proof}
Defining
$$
\psi_n(\on) = \sum_{e \in \B_n} (\phi_n^e - \phi_n\z)(0,\on),
$$
we note that, in close resemblance with \eqref{defa1zmu},
\begin{multline}
\label{defa1zn}
a_1\z(n)=  \sum_{e : \un{e} = 0} \Ll(\E_n[\xi \cdot A^e(\xi + \na {\phi}_n^e)] - \E_n[\xi \cdot A\z(\xi + \na {\phi}_n\z)] \Rr) + \\
\E_n[\xi \cdot A\z D\psi_n] - \sum_{e : \un{e} = 0}\E_n[\xi \cdot A\z(\na \phi^e_n - \na \phi\z_n)].
\end{multline}
It thus suffices to prove the following three statements.
\begin{enumerate}
	\item The function $\un{\o} \mapsto \na \phi_n\z(0,\on)$ converges in $L^2(\un{\O})$ to $\na \phi\z(0,\cdot)$ as $n$ tends to infinity.
	\item The function $\un{\o} \mapsto \na \phi_n^e(0,\on)$ converges in $L^2(\un{\O})$ to $\na \phi^e(0,\cdot)$ as $n$ tends to infinity.
	\item The function $\un{\o} \mapsto D\psi_n(\on)$ converges in $L^2(\un{\O})$, as $n$ tends to infinity, to the function $\chi$ defined in Proposition~\ref{p:defcoef}.
\end{enumerate}
Part (1) is classical (see for instance \cite{capiof, boupia}), but it will be useful to recall a proof. We begin by observing that the function $\phi_n\z$ is stationary, in the sense that for every $\un{\o} \in \un{\O}$ and every $x \in \Z^d$, we have $\phi_n\z(x,\on) = \phi_n\z(0,\theta_x \ \on)$. As a consequence, we may identify $\phi_n\z$ with the function $\phi_n\z(0,\cdot)$. We then observe that
\begin{equation}
\label{period-energy}
\E_n\Ll[ D\phi_n\z \cdot A\z D\phi_n\z \Rr] = -\E_n\Ll[ D\phi_n\z \cdot A\z \xi \Rr],
\end{equation}
and thus, by the Cauchy-Schwarz inequality,
$$
c_- \E_n\Ll[|D\phi_n\z|^2\Rr] \le \E_n\Ll[ D\phi_n\z \cdot A\z D\phi_n\z \Rr] \le c_+ |\xi| \  \E_n\Ll[|D\phi_n\z|^2\Rr]^{1/2}.
$$
It follows from this inequality that $\E_n\Ll[|D\phi_n\z|^2\Rr]$ is bounded uniformly over $n$, and thus that $\un{\o} \mapsto D\phi_n\z(0,\on)$ converges weakly in $L^2(\un{\O})$ along a subsequence, to some $\td{\chi} \in L^2(\un{\O})$. For notational simplicity, we keep implicit the fact that we now consider $(D\phi_n\z)$ only along this subsequence. In order to prove part (1), it suffices to show that the weak convergence is actually strong convergence in $L^2(\un{\O})$, and that $\td{\chi} = \na \phi\z$. 

Note first that \eqref{period-energy} implies that (along the subsequence)
\begin{equation}
\label{csqenergy}
\lim_{n \to \infty} \E_n\Ll[ D\phi_n\z \cdot A\z D\phi_n\z \Rr] = -\E\Ll[\td{\chi} \cdot A\z \xi\Rr].
\end{equation}
Now, take any bounded function $f : \un{\O} \to \R$ that depends on the value of the environment only at a finite number of edges
. We have
$$
\E_n\Ll[Df \cdot A\z D\phi_n\z \Rr] = -\E_n\Ll[Df \cdot A\z \xi\Rr].
$$
Passing to the limit (along the subsequence), we get that
\begin{equation}
\label{weaktdchi}
\E\Ll[Df \cdot A\z \td{\chi} \Rr] = -\E\Ll[Df \cdot A\z \xi\Rr].
\end{equation}
The identity above can then be extended to arbitrary $f \in L^2(\un{\O})$. Replacing $f$ by $\phi_n\z$ in the above identity, and letting $n$ tend to infinity, we thus learn that
\begin{equation}
\label{weaktdchi2}
\E\Ll[\td{\chi} \cdot A\z \td{\chi} \Rr] = -\E\Ll[\td{\chi} \cdot A\z \xi\Rr].
\end{equation}
To see strong convegence of $D\phi_n$ to $\td{\chi}$ in $L^2(\un{\O})$, it now suffices to write 
\begin{multline*}
\E_n\Ll[(D\phi_n - \td{\chi}) \cdot A\z (D\phi_n - \td{\chi}) \Rr] \\ = \E_n\Ll[D\phi_n \cdot A\z D\phi_n \Rr] + \E_n\Ll[\td{\chi} \cdot A\z \td{\chi} \Rr] - 2 \E_n\Ll[D\phi_n \cdot A\z \td{\chi} \Rr],
\end{multline*}
and observe that the right-hand side tends to zero by \eqref{csqenergy} and \eqref{weaktdchi2}. This ensures that $\td{\chi} \in L^2_\na$, and by \eqref{weaktdchi}, $\td{\chi}$ satisfies 
$$
-D^* \cdot A\z \td{\chi} = D^* \cdot A\z \xi.
$$
By Theorem~\ref{t:station}, there is a unique such $\td{\chi}$, and it is $\na \phi^\circ$ (see Proposition~\ref{p:corrlim}).

Part (3) can be obtained in a similar way, noting that the function $\psi_n$ is stationary, in the sense that $\psi_n(x,\on) = \psi_n(0, \theta_x \ \on)$.

For part (2), we cannot argue in the very same way, since $A^e$ is not stationary. We note instead that, letting $\ov{\phi}_n^e = \phi_n^e - \phi_n\z$, we have
$$
-\na^* \cdot A^e \na \ov{\phi}_n^e = \na^* \cdot C^e (\xi + \na \phi_n\z)
$$
at every point $(x,\on)$, with $x \in \Z^d$ and $\un{\o} \in \un{\O}$. We can then simply quote \cite[Lemma~A.2]{analeb2}, or argue as in part (1), with the difference that the arguments need to be carried out over the physical space $\Z^d$ instead of the space of environments. We get that $(x,\un{\o}) \mapsto \1_{B_n} \na \ov{\phi}_n^e(x,\on)$ converges in $L^2(\Z^d \times \un{\O})$ to $\na\ov{\phi}_\infty^e \in L^2(\Z^d\times \un{\O})$ as $n$ tends to infinity, where $\na \ov{\phi}_\infty^e$ satisfies
\begin{equation}
\label{e:ananth}	
-\na^* \cdot A^e \na \ov{\phi}_\infty^e = \na^* \cdot C^e (\xi + \na \phi\z) \quad  (\text{in } \Z^d \times \un{\O})
\end{equation}
(and the space $L^2(\Z^d \times \un{\O})$ is defined with respect to the measure obtained as the product of the counting measure over $\Z^d$ and $\P$). This identifies $\na \ov{\phi}_\infty^e$ as
$$
\na \ov{\phi}_\infty^e(x) = - \Ll(\na_1 \na_2 G^e(x,\un{e})\Rr)  C^e(\un{e}) (\xi + \na \phi\z(\un{e})).
$$
Comparing this with the formula for $\na \phi^e$ given in \eqref{e:naphie}, we see that we are done.
\end{proof}

The question arises as to whether the sum involving all $(\phi_n^e)_{e \in \B_n}$ in the formula defining $a\z_1(n)$ can be replaced by a spatial average involving only $(\phi_n^e)_{e :\un{e} = 0}$. This question is important since in practice, one would like to compute as few $\phi_n^e$'s as possible. We define
\begin{equation}
\label{defova}
\ov{a}_n(\on) = \sum_{x \in B_n} \sum_{e : \un{e} = 0} \Big( \xi \cdot A^{e}(\xi + \na \phi_n^e) -  \xi \cdot A\z(\xi + \na \phi_n\z) \Big)(x,\on).
\end{equation}
\begin{prop}[approximation of $a_1\z$ by periodization and spatial average]
\label{p:spatial-average}
The random variable $\un{\o} \mapsto \ov{a}_n(\on)$ converges in $L^1(\un{\O})$ to the random variable
$$
\sum_{e : \un{e} = 0} (\xi + \na \phi\z) \cdot C^e(\xi + \na \phi^e) \ (0,\cdot).
$$
Moreover, the coefficient $a_1\z$ given by Theorem~\ref{t:main0} satisfies
\begin{equation*}
a_1\z = \sum_{e : \un{e} = 0} \E\Ll[(\xi + \na \phi\z) \cdot C^e(\xi + \na \phi^e) \Rr] = \sum_{e : \un{e} = 0} \lim_{n \to \infty} \E_n\Ll[ (\xi + \na \phi_n\z) \cdot C^e(\xi + \na \phi_n^e)  \Rr].
\end{equation*}
\end{prop}
\begin{proof}
Let $e$ be such that $\un{e} = 0$. 
By the definition of $\phi_n^e$ and $B_n$-periodicity,
$$
\sum_{x \in B_n} \xi \cdot A^e(\xi + \na \phi_n^e)(x,\on) = \sum_{x \in B_n} (\xi + \na \phi_n\z) \cdot A^e(\xi + \na \phi_n^e)(x,\on).
$$
We can decompose $A^e$ into $A\z + C^e$, and then observe that the definition of $\phi_n\z$ implies that
$$
\sum_{x \in B_n} (\xi + \na \phi_n\z) \cdot A\z(\xi + \na \phi_n^e)(x,\on) = \sum_{x \in B_n} (\xi + \na \phi_n\z) \cdot A\z \xi (x,\on).
$$
Since $A\z$ is a symmetric matrix, we have obtained that
\begin{equation}
\label{defovabis}
\ov{a}_n(\on) = \sum_{e : \un{e} = 0}(\xi + \na \phi_n\z) \cdot C^e(\xi + \na \phi_n^e)(0,\on).
\end{equation}
The convergence announced in the proposition then follows from the fact that $\un{\o} \mapsto \na \phi_n\z(0,\on)$ and $\un{\o} \mapsto \na \phi_n^e(0,\on)$ converge in $L^2(\un{\O})$ to $\na \phi\z(0,\cdot)$ and $\na \phi^e(0,\cdot)$ respectively. The second part of the proposition follows noting that $\E_n[\ov{a}_n] = a_1\z(n)$, see \eqref{defa1n}.
\end{proof}

%
%
%
%
%
%
%
%
\section{First-order expansion around every point}
\label{s:expansion}

The aim of this section is to generalize Theorem~\ref{t:main0} by giving a first-order expansion of $\Ah\ppp$ as $p$ tends to $0$, for every $\ov{p} \in [0,1]$. 

In order to state the result, it is convenient to introduce some new notation. For every $\ov{p} \in [0,1]$ and every $e \in \B$, we let $A\ppep = A^{E\pp \cup \{e\}}$ (recall that $A^E$ was defined in \eqref{defAE}), $\phi_\mu\ppep = \phi_\mu^{E\pp \cup \{e\}}$ (defined in \eqref{defphiE}), and so on. Similarly, we define $A\ppem = A^{E\pp \setminus \{e\}}$,  $\phi_\mu\ppem = \phi_\mu^{E\pp \setminus \{e\}}$, and so on.

\begin{thm}[first-order expansion around every point]
\label{t:main}
Let $\ov{p} \in [0,1]$. There exists $a_1\pp \in \R$ such that, as $p$ tends to $0$,
\begin{equation}
\label{e:main}
\xi \cdot \Ah\ppp \xi = \xi \cdot \Ah\pp \xi + p \ a_1\pp + o(p).
\end{equation}
Moreover, the coefficient $a_1\pp$ can be defined as the limit of $a_1\pp(\mu)$ as $\mu$ tends to $0$, where $a_1\pp(\mu)$ is given by
\begin{equation}
\label{defa1mupp}
a_1\pp(\mu) = \sum_{e \in \B} \Ll(\E[\xi \cdot A\ppep(\xi + \na {\phi}_\mu\ppep)] - \E[\xi \cdot A\ppem(\xi + \na {\phi}_\mu\ppem)] \Rr).
\end{equation}
Alternative characterizations of $a_1\pp$ are given in Remark~\ref{r:1} below.
\end{thm}

One may attempt to prove this result by defining a coupling of the Bernoulli random variables such that if $p_1 \le p_2$, then $E^{(p_1)} \subset E^{(p_2)}$, and then try to adapt the proof of Theorem~\ref{t:main0}. This approach leads to a serious difficulty when trying to estimate the left-hand side of \eqref{e:thestep}. Indeed, this estimate crucially relies on the fact that the non-perturbed environment is independent of the events $e \in E_\mu\p$. A naive adaptation of this proof to cover the expansion of $\Ah\ppp$ as $p$ tends to $0$ would thus require that the environment $\o\pp$ be independent of the events $e \in E_\mu\ppp \setminus E_\mu\pp$, and of course, this independence does not hold. 

This problem can be overcome with a control of higher moments of the gradients of the Green functions than that given by Proposition~\ref{p:greenaverage}. Such results have been obtained very recently in \cite{marott}. We will describe and use these results in the next section, but for the problem at hand, a more direct proof can be obtained.

\begin{proof}[Proof of Theorem~\ref{t:main}]
It comes out of the above discussion that the decision as to which edges are to be perturbed should be taken independently of the environment. We thus introduce new independent Bernoulli random variables $(\td{b}_e\p)_{e \in \B}$ of parameter $p \in [0,1]$, independent of everything else, defined on the same probability space (possibly enlarging it)
. For $\ov{p} < 1$ and $p \in [0,1-\ov{p}]$, we further define
$$
\td{p} = \frac{p}{1-\ov{p}} \qquad \text{and} \qquad \td{\o}_e\ppvp = (1- \td{b}_e\tp) \ \omega\pp_e + \td{b}_e\tp \ \omega^{(1)}.
$$
Recalling that we write $\nu\p$ for the law of $\o_e\p$, we see that the law of $\td{\o}_e\ppvp$ is 
\begin{multline*}
(1-\td{p}) \nu\pp + \td{p} \nu^{(1)} = (1-\td{p})\Ll[ (1-\ov{p}) \nu^{(0)} + \ov{p} \nu^1 \Rr] + \td{p} \nu^{(1)} \\ = (1-\ov{p} - p) \nu^{(0)} + (\ov{p} + p) \nu^{(1)} = \nu\ppp.
\end{multline*}
In other words, the law of $\td{\o}_e\ppvp$ is that of $\o_e\ppp$, and as a consequence, the law of $\td{\o}\ppvp := (\td{\o}_e\ppvp)_{e \in \B}$ is that of $\o\ppp$.

We can apply Theorem~\ref{t:main0} for the environment $\td{\o}\ppvp$, seen as a Bernoulli perturbation of $\o\pp$, with the perturbed environment taken from $\o^{(1)}$ and the perturbation parameter being $\td{p}$. Since the homogenized matrix of $\td{\o}\ppvp$ is that of $\o\ppp$, we get that as $\td{p} \ge 0$ tends to $0$,
$$
\xi \cdot \Ah\ppp \xi = \xi \cdot \Ah\pp \xi + \td{p} \ \td{a}_1\pp + o(\td{p}),
$$
where $\td{a}_1\pp$ is the limit as $\mu$ tends to $0$ of $\td{a}_1\pp(\mu)$ given by
$$
\td{a}_1\pp(\mu) = \sum_{e \in \B} \Ll(\E[\xi \cdot A\ppep(\xi + \na {\phi}_\mu\ppep)] - \E[\xi \cdot A\pp(\xi + \na {\phi}_\mu\pp)] \Rr).
$$
We now see that the quantity
$$
\xi \cdot A\ppep(\xi + \na {\phi}_\mu\ppep) - \xi \cdot A\pp(\xi + \na {\phi}_\mu\pp)
$$
is equal to $0$ if $e \in E\pp$, and otherwise is equal to 
$$
\xi \cdot A\ppep(\xi + \na {\phi}_\mu\ppep) - \xi \cdot A\ppem(\xi + \na {\phi}_\mu\ppem).
$$
Since this last quantity is independent of the event $e \in E\pp$, we obtain that
\begin{multline}
\E\Ll[\xi \cdot A\ppep(\xi + \na {\phi}_\mu\ppep) - \xi \cdot A\pp(\xi + \na {\phi}_\mu\pp)\Rr] \\
= \P\Ll[e \notin E\pp\Rr] \ \E\Ll[ \xi \cdot A\ppep(\xi + \na {\phi}_\mu\ppep) - \xi \cdot A\ppem(\xi + \na {\phi}_\mu\ppem) \Rr].
\end{multline}
Since $\P\Ll[e \notin E\pp\Rr] = 1-\ov{p}$, we obtain the desired result for $p \ge 0$ (i.e. for the right derivative of $\xi \cdot \Ah^{(\cdot)} \xi$). 

A similar reasoning can be followed if $\ov{p} \in (0,1]$ and $p \in [-\ov{p},0]$, by letting
$$
\td{p} = \frac{-p}{\ov{p}} \qquad \text{and} \qquad \td{\o}_e\ppvp = (1- \td{b}_e\tp) \ \omega\pp_e + \td{b}_e\tp \ \omega^{(0)}.
$$
The reasoning leads to the same formula for the left derivative, so the proof is complete.
\end{proof}

\begin{rem}
\label{r:1}
From the construction in the proof of Theorem~\ref{t:main} and Proposition~\ref{p:periodiz}, one can give an alternative characterization of the coefficient $a_1\pp$ appearing in Theorem~\ref{t:main}. Indeed, it is the limit of $a_1(n)$ as $n$ tends to infinity, for
$$
a_1(n) = \sum_{e \in \B_n} \Ll( \E_n\Ll[\xi \cdot A\ppep(\xi + \na \phi_n\ppep)\Rr] - \E_n\Ll[\xi \cdot A\ppem(\xi + \na \phi_n\ppem)\Rr]\Rr),
$$
where $\phi_n\ppep$ and $\phi_n\ppem$ are defined as $\phi_n\z$, but with respect to $A\ppep$ and $A\ppem$ respectively instead of $A\z$. 

Similarly, using Proposition~\ref{p:spatial-average}, we can obtain yet another characterization of $a_1\pp$, namely
\begin{equation}
\label{e:a1ppe}
a_1\pp =\sum_{e : \un{e} = 0} \E\Ll[ (\xi + \na \phi\ppem) \cdot C^e (\xi + \na \phi\ppep)\Rr],
\end{equation}
where $\na \phi\ppem(0,\cdot)$ and $\na \phi\ppep(0,\cdot)$ are the $L^2(\un{\O})$-limits of $\na \phi_n\ppem(0,\cdot)$ and $\na \phi_n\ppep(0,\cdot)$ respectively, as $n$ tends to infinity. Moreover,
\begin{equation}
\label{e:a1ppbis}
a_1\pp = \sum_{e : \un{e} = 0} \lim_{n \to \infty} \E_n\Ll[ (\xi + \na \phi_n\ppem) \cdot C^e (\xi + \na \phi_n\ppep) \Rr].
\end{equation}
\end{rem}

%
%
%
%
%
%
%
%
\section{Refined control of the error term}
\label{s:error}

The aim of this section is to obtain a sharper control of the error term $o(p)$ appearing in \eqref{e:main} of Theorem~\ref{t:main}. As is clear from the proof of Theorem~\ref{t:main}, in order to achieve this, it is in fact sufficient to control the error term in Theorem~\ref{t:main0}. The fact that Theorem~\ref{t:main0} gives only a weak control of the error term lies in Theorem~\ref{t:approxcorr}. The weak point in the argument leading to Theorem~\ref{t:approxcorr} is in the estimation of the left-hand side of \eqref{e:thestep}, where we crudely estimated the second mixed derivatives of $G_\mu\p$ using the pointwise inequality \eqref{greenpoint2} on the gradient of $G_\mu\p$. Although Proposition~\ref{p:greenaverage} shows that the $L^1$ norm of the second mixed derivative of $G_\mu\p$ has the same spatial decay as the homogeneous case, this is of no use when estimating the left-hand side of \eqref{e:thestep} because we have no control on the dependence between the second derivatives of $G_\mu\p$ and the events $e \in E_\mu\p$. 

This problem can however be overcome with a control of higher moments of the second derivatives of the Green function. Such results have been obtained very recently in the remarkable work \cite{marott}.
\begin{thm}[\cite{marott}]
\label{t:marott}
For every $q \ge 1$, there exists a finite $\C_q$ depending only on $q$, $d$ and the ellipticity constants $c_-, c_+$ (but otherwise not on the law of $\omega$, as long as it is a product measure as assumed throughout) such that for every $x \in \Z^d$ and $\mu \ge 0$,
\begin{equation}
\label{e:marott}
\E\Ll[|\na_1 \na_2 G_\mu(0,x)|^{q}\Rr]^{1/q} \le \frac{\C_q}{(1 \vee |x|)^d}.
\end{equation}
\end{thm}
\begin{rem}
Striclty speaking, \cite[Theorem~1]{marott} corresponds to this result for $\mu = 0$. The proof can however be adapted with minor modifications to yield Theorem~\ref{t:marott}. Let us describe briefly how. \cite[Lemma~4]{marott} needs no change. \cite[Lemma~6]{marott} consists of two parts. The proof of the first part can be easily adapted to yield the result with $G$ replaced by $G_\mu$ (actually, it provides an alternative route to the periodization argument used there), while the second part was proved for $G_\mu$ in the first place (that is, in \cite{glotto}). The first part of \cite[Lemma~5]{marott} is all what is needed for our purpose, and it remains true with $G$ replaced by $G_\mu$ (uniformly over $\mu$). Indeed, the formulas for the derivatives of the Green function with respect to $\omega_e$ appearing in steps 1 and 2 of the proof, as e.g. \cite[(48)-(50)]{marott}, remain valid provided $G$ is replaced by $G_\mu$ everywhere. Step 3 is a consequence of \cite[Lemma~6]{marott}, while step 4 follows from previous steps, and concludes the proof of the part of the lemma we are interested in. Finally, the proof of \cite[Theorem~1]{marott} follows from the results we just reviewed (again providing an alternative route to the periodization argument proposed there). The fact that the constant $\C_q$ in \eqref{e:marott} can be chosen as a function of $q$, $d$ and the ellipticity constants only boils down to the fact that the constant $\rho$ in the weak logarithmic Sobolev inequality defined in \cite[(2)]{marott} can be chosen as a function of $c_-$ and $c_+$ only.
\end{rem}

\begin{thm}[sharp control of the error]
\label{t:error}
Let $\ov{p} \in [0,1]$ and $a_1\pp \in \R$ be given by Theorem~\ref{t:main}. For every $\eta >0$ and as $p$ tends to $0$, 
\begin{equation}
\label{e:error}
\xi \cdot \Ah\ppp \xi = \xi \cdot \Ah\pp \xi + p \ a_1\pp + o(p^{2-\eta}).
\end{equation}
\end{thm}

As was said before, it suffices to prove Theorem~\ref{t:error} for $\ov{p} = 0$, and the key step for doing so is to prove the following result.

\begin{thm}[sharp linear approximation of the corrector] 
\label{t:sharp}
For every $\eta > 0$, there exist constants $c$, $\gamma > 0$ such that for every $\mu \in (0,1/2]$ and $p \in [0,1]$,
$$
\E\Ll[\Ll|\na \phi\p_\mu - \na \td{\phi}\p_\mu \Rr|\Rr] \le c p^{2-\eta} \mu^{-\eta} + e^{-\mu^{-\gamma}}.
$$
\end{thm}
We start with a result whose purpose is to replace Lemma~\ref{l:greenperturbed}.
\begin{lem}[sharp averaged control of the perturbed Green function]
There exists a finite constant $c$ (depending only on the ellipticity constants and on the dimension) such that for every $q \ge 1$, $\mu \in (0,1/2]$ and $e,e' \in \B$,
$$
\E\Ll[ | \na_1 \na_2 G^e_\mu(\un{e}',\un{e})|^{q}  \Rr]^{1/q} \le c \   \frac{\C_q}{(1\vee |\un{e}-\un{e}'|)^d},
$$
where $\C_q$ is the constant of Theorem~\ref{t:marott}.
\label{l:sharpgreenperturbed}
\end{lem}
\begin{proof}
Recall from the proof of Lemma~\ref{l:greenperturbed} that  $\ov{G}^e_\mu = G^e_\mu - G\z_\mu$. Since by Theorem~\ref{t:marott},
\begin{equation}
\label{e:sharpgreen}
\E\Ll[ | \na_1 \na_2 G\z_\mu(\un{e}',\un{e})|^{q}  \Rr]^{1/q} \le \frac{\C_q}{(1\vee |\un{e}-\un{e}'|)^d},
\end{equation}
it suffices to prove that
$$
\E\Ll[ | \na_1 \na_2 \ov{G}^e_\mu(\un{e}',\un{e})|^{q}  \Rr]^{1/q} \le c \ \frac{\C_q}{(1\vee |\un{e}-\un{e}'|)^d}.
$$
Recall from \eqref{relovG} that
$$
\na_1 \na_2\ov{G}^e_\mu(\un{e}',\un{e}) = C^e(\un{e}) \Ll(\na_1 \na_2 G^e_\mu(\un{e},\un{e})\Rr) \ \Ll(\na_1 \na_2G\z_\mu(\un{e}',\un{e})\Rr).
$$
Since, by \eqref{greenpoint2},
$$
\Ll|  C^e(\un{e}) \na_1 \na_2 G^e_\mu(\un{e},\un{e}) \Rr|
$$
is bounded uniformly over $\mu$ and $\un{\o}$, the result follows from \eqref{e:sharpgreen}.
\end{proof}
\begin{proof}[Proof of Theorem~\ref{t:sharp}]
For the beginning of the proof, we can follow the same reasoning as in the proof of Theorem~\ref{t:approxcorr}, up to the point when we arrive at the estimation of the left-hand side of \eqref{e:thestep}, that is,
$$
\E\Ll[ \sum_{e\neq e' \in  E\p_\mu}  | \na_1 \na_2 G\p_\mu(0,\un{e}')| \ | \na_1 \na_2 G^e_\mu(\un{e}',\un{e})| \ |\xi + \na \phi\z_\mu(\un{e})| \Rr],
$$
where we recall that $E_\mu\p = E\p \cap B_{\mu^{-\gamma}}$, and $\gamma > 1/2$. We rewrite slightly this expectation as
\begin{equation}
\label{err1}
\sum_{e \neq e' \in B_{\mu^{-\gamma}}} \E\Ll[  \1_{e,e' \in E\p} | \na_1 \na_2 G\p_\mu(0,\un{e}')| \ | \na_1 \na_2 G^e_\mu(\un{e}',\un{e})| \ |\xi + \na \phi\z_\mu(\un{e})| \Rr].
\end{equation}
Let $q > 3$ and $\zeta=1-3/q$. We apply Hölder's inequality with exponents $(\zeta^{-1},q,q,q)$ to bound each summand in \eqref{err1} by
\begin{multline}
\label{err2}
\P[e,e' \in E\p]^{\zeta} \ \E\Ll[ | \na_1 \na_2 G\p_\mu(0,\un{e}')|^{q}\Rr]^{1/q} \\
\E\Ll[ \ | \na_1 \na_2 G^e_\mu(\un{e}',\un{e})|^{q}\Rr]^{1/q} \ \E\Ll[ |\xi + \na \phi\z_\mu(\un{e})|^{q}\Rr]^{1/q}.
\end{multline}
Since we only consider summands for which $e \neq e'$, we have
$$
\P[e,e' \in E\p]^{\zeta} = p^{2 \zeta}.
$$
By Theorem~\ref{t:marott} and Lemma~\ref{l:sharpgreenperturbed}, the product of the second and third terms in \eqref{err2} is bounded by
$$
c \ \frac{(\C_q)^2}{(1 \vee |\un{e}'|)^d \ (1 \vee |\un{e} - \un{e}'|)^d},
$$
while, by Theorem~\ref{t:approx}, the last term in \eqref{err2} is bounded by some power of $\log(\mu^{-1})$, say $\log^r(\mu^{-1})$. We thus obtain that the sum in \eqref{err1} is bounded by
$$
c \ p^{2\zeta} \log^r(\mu^{-1})  \sum_{e \neq e' \in B_{\mu^{-\gamma}}} \frac{(\C_q)^2}{(1 \vee |\un{e}'|)^d \ (1 \vee |\un{e} - \un{e}'|)^d}.
$$
By comparing this sum to an integral, we get that the whole expression above is bounded, up to a constant, by 
$p^{2\zeta} \mu^{-\gamma \eta}$,
for any $\eta > 0$. Recalling that $\zeta = 1-3/q$ and that $q$ can be chosen arbitrarily large, this concludes the proof.
\end{proof}
\begin{proof}[Proof of Theorem~\ref{t:error}] As noted before, it suffices to prove the theorem for $\ov{p} = 0$. The proof follows closely that of Theorem~\ref{t:main0}. Let $\eta > 0$. Instead of \eqref{fixmu}, we fix $\mu = p^2$. With this choice, Theorem~\ref{t:approx} implies that the left-hand side of \eqref{e:approx} is $O(p^{2-\eta})$, while by Theorem~\ref{t:sharp}, the left-hand side of \eqref{e:approx2} is $O(p^{2-3\eta})$. We can then follow the proof of Theorem~\ref{t:main0} without any change until \eqref{e:presque}, whose error term is now $O(p^{2-3\eta})$ instead of $o(p)$. The proof is then concluded in the same way, noting that Theorem~\ref{t:approx} ensures that $\E[\xi \cdot A\z(\xi + \na {\phi}_\mu\z)] = \xi \cdot \Ah\z \xi + O(p^{2-\eta})$.
\end{proof}

\bigskip

\noindent \textbf{Acknowledgements.} I would like to warmly thank Vincent Beffara, Noam Berger, Antoine Gloria, Claude Le Bris and Frédéric Legoll for stimulating discussions about this problem.

\end{document}